\documentclass[a4paper,11pt]{amsart}
\usepackage{amsmath,amssymb,amsfonts,amsthm}

\usepackage[alphabetic]{amsrefs}
\usepackage{bm}
\usepackage{graphicx}
\usepackage{ascmac}
\usepackage[all]{xy}
\usepackage{amsthm}
\usepackage[top=25mm,bottom=25mm,left=30mm,right=30mm]{geometry}
\usepackage{tikz}
\usetikzlibrary{patterns}
\usepackage{pxpgfmark}
\usepackage{multirow}
\usetikzlibrary{intersections,arrows,calc,decorations.markings}
\newtheorem{theorem}{Theorem}[section]
\newtheorem{proposition}[theorem]{Proposition}
\newtheorem{conjecture}[theorem]{Conjecture}
\newtheorem{corollary}[theorem]{Corollary}
\newtheorem{lemma}[theorem]{Lemma}

\theoremstyle{definition}
\newtheorem{remark}[theorem]{Remark}
\newtheorem{example}[theorem]{Example}
\newtheorem{definition}[theorem]{Definition}

\makeatletter
 
 \@addtoreset{equation}{section}
\makeatother
\def\bb{\mathbf{b}}
\def\cc{\mathbf{c}}
\def\dd{\mathbf{d}}
\def\ee{\mathbf{e}}
\def\ff{\mathbf{f}}
\def\gg{\mathbf{g}}

\def\vv{\mathbf{v}}
\def\mm{\mathbf{m}}
\def\ww{\mathbf{w}}
\def\xx{\mathbf{x}}
\def\yy{\mathbf{y}}
\def\TT{\mathbb{T}}

\def\PP{\mathbb{P}}
\def\ZZ{\mathbb{Z}}
\def\QQ{\mathbb{Q}}
\def\Acal{\mathcal{A}}
\def\Fcal{\mathcal{F}}
\def\Xcal{\mathcal{X}}

\def\QQsf{\mathbb{Q}_{\text{sf}}}

\newcommand{\opname}[1]{\operatorname{\mathsf{#1}}}

\newcommand{\Hom}{\opname{Hom}}
\newcommand{\End}{\opname{End}}
\newcommand{\go}{\opname{G_0}}

\newcommand{\add}{\opname{add}\nolimits}
\newcommand{\ie}{{\em i.e.}\ }

\newcommand{\id}{\mathbf{1}}
\newcommand{\im}{\opname{im}\nolimits}

\title[Compatibility degree of cluster complexes]{Compatibility degree of cluster complexes}
\author{Changjian Fu}
\author{Yasuaki Gyoda}
\address{Changjian Fu: Department of Mathematics, Sichuan University, Chengdu, 610064 PR China}
\email{changjianfu@scu.edu.cn}
\address{Yasuaki Gyoda: Graduate School of Mathematics, Nagoya University, Chikusa-ku, Nagoya, 464-8602 Japan}
\email{m17009g@math.nagoya-u.ac.jp}
\begin{document}
\begin{abstract}
We introduce a new function on the set of pairs of cluster variables via $f$-vectors, which is called the compatibility degree (of cluster complexes). The compatibility degree is a natural generalization of the classical compatibility degree introduced by Fomin and Zelevinsky. In particular, we prove that the compatibility degree has the duality property, the symmetry property, the embedding property and the compatibility property, which the classical one has. We also conjecture that the compatibility degree has the exchangeability property. As pieces of evidence of this conjecture, we establish the exchangeability property for cluster algebras of rank 2, acyclic skew-symmetric cluster algebras, cluster algebras arising from weighted projective lines, and cluster algebras arising from marked surfaces.
\end{abstract}
\maketitle
\section{Introduction}
\emph{Cluster algebras} were defined in \cite{fzi} to study the dual canonical bases and the total positivity in semisimple algebraic groups at first. They are commutative algebras generated by the \emph{cluster variables}. These generators are gathered into overlapping sets of fixed finite cardinality, called {\it clusters}, which are defined recursively from an initial one via {\it mutation}. The {\it cluster complex} of a cluster algebra is the simplicial complex whose simplexes are subsets of cluster variables in each cluster \cite{fzii} . It encodes the mutation of cluster variables and has played an important role in the study of cluster algebras and their interaction with different mathematical subjects. For example, \cite{bmrrt} showed that a simplicial complex whose simplexes are subsets of tilting sets in a cluster category coincides with one of cluster complexes. This identification yields an application of representation theory to cluster algebras. Also, \cite{fst} pointed out that triangulations of marked surfaces and their flips have cluster structures by proving that arc complexes correspond with cluster complexes. This indicates a connection of cluster algebras with hyperbolic geometry.

A {\it compatibility degree} of a cluster complex (or cluster algebra) is a function on the set of pairs of cluster variables satisfying various properties. Such a function was first introduced by Fomin and Zelevinsky \cite{fzy} for {\it generalized associhedra} associated with finite root systems in their study of Zamolodchikov's periodicity for $Y$-systems, which is a special kind of cluster complexes of cluster algebras. In particular, for each finite root system $\Phi$, there is a cluster algebra $\mathcal{A}(\Phi)$ whose cluster complex coincides with the generalized associahedron associated with $\Phi$. 
Let $\Phi_{\geq -1}$ be the set of almost positive roots of $\Phi$, {\it i.e.} the union of all negative simple roots and all positive roots.
 In this case, there is a bijection between the set of cluster variables of $\mathcal{A}(\Phi)$ and $\Phi_{\geq -1}$ given by the denominator vectors ($d$-vectors) of cluster variables~\cites{fzii,fziv}. 
Under this bijection, the (classical) compatibility degree of $\mathcal{A}(\Phi)$ introduced in~\cite{fzy} is a function $(\cdot\parallel\cdot)_{\rm{cl}}$ from pairs of almost positive roots to nonnegative integers. 
{The classical compatibility degree $(\cdot\parallel\cdot)_{\rm{cl}}$ was generalized by Reading \cite{rea}, who introduced the $c$-compatibility degree for not necessary crystallographic root systems. Moreover, Ceballos and Pilaud shows that $c$-compatibility degree is also given by using $d$-vectors in \cite{cp}.} Recently, Cao and Li~\cite{cl2} introduced a compatibility degree (in this paper, it is called the $d$-compatibility degree) for any cluster complexes by using $d$-vectors. However, the $d$-compatibility degree does not preserve many properties of the classical one.

The main theme of this paper is to introduce a new compatibility degree for any cluster complexes (Definition \ref{f-compatibility} and Theorem \ref{classical-fcorrespondence}). Instead of using $d$-vectors as in \cite{cl2}, we propose to generalize the classical compatibility degree by using the $f$-vectors. We simply call it the \emph{compatibility degree}. The $f$-vectors are maximal degree vectors of $F$-polynomials, which are polynomials of coefficients in cluster variables. They were studied in \cites{fziv,fk,fg,y,gy,g}. It is conjectured that $F$-matrices, which consist of $f$-vectors in a cluster, determine a cluster uniquely. 

{Our main result is that the compatibility degree defined by the $f$-vectors gives a more ``natural" generalization of the classical compatibility degree than the one defined by the $d$-vectors.}
First, the compatibility degree has the \emph{duality property}, the \emph{symmetry property} and the \emph{embedding property} (Proposition \ref{pr:f-compatibility-symmetry}). In classical cases, the degree of $\alpha$ and $\beta$ in $\Phi_{\geq-1}$ equals to that of coroots of $\beta$ and $\alpha$ in the dual root system $\Phi_{\geq-1}^{\vee}$, that is, $(\alpha\parallel\beta)_{\rm{cl}}=(\beta^\vee\parallel\alpha^\vee)_{\rm{cl}}$. This is the duality property. The compatibility degree has the same property for a cluster complex and its dual cluster complex. Next, we give the symmetry property. If $\Phi$ is simply laced, then the compatibility degree is symmetric, that is, $(\alpha\parallel\beta)_{\rm{cl}}=(\beta\parallel\alpha)_{\rm{cl}}$. If $\Phi$ is non simply laced, then this equation does not hold. However, if $(\alpha\parallel\beta)_{\rm{cl}}=0$, then $(\beta\parallel\alpha)_{\rm{cl}}=0$. By using the compatibility degree, we generalize this property to cluster complexes. Moreover, we give a relation between $(\alpha\parallel\beta)_{\rm{cl}}$ and $(\beta\parallel\alpha)_{\rm{cl}}$ in non-simply laced cases.
 The embedding property is the property that the degree of $\alpha$ and $\beta$ in $\Phi_{\geq-1}$ equals that in a root subsystem $\Phi'_{\geq-1}$ of $\Phi_{\geq-1}$. Considering a cluster complex and their cluster subcomplex, we can generalize this property.
 
 Second, the compatibility degree has the \emph{compatibility property} (Theorem \ref{fcompatibilityproperty}), that is, for any cluster variables $x$ and $x'$, there is a simplex which contains both $x$ and $x'$ if and only if the compatibility degree of $x$ and $x'$ is $0$. This property implies that the $f$-vectors determine the cluster complex. 
 
 In this paper, we also discuss the \emph{exchangeability property}. In classical cases, it is known that for distinct positive roots $\alpha$ and $\beta$, there exists a subset $X$ of almost positive roots such that $X\cup{\alpha}$ and $X\cup{\beta}$ are both maximal simplexes of a generalized associahedron if and only if $(\alpha\parallel\beta)_{\rm{cl}}=(\beta\parallel\alpha)_{\rm{cl}}=1$ (\cites{cfz,fzii}). We consider whether the compatibility degree satisfies an analogy of this property. We prove the ``only if" part for the general cluster complex case (Theorem \ref{f-exchangeability1}). However, the ``if" part is still open (Conjecture \ref{f-exchangeability2}). We solve this problem partially by a description of $F$-matrices of cluster algebras of rank 2 (Theorem \ref{f-exchangeability2rank2}) and by using 2-Calabi-Yau categorification (Theorems \ref{t:conj-acyclic}, \ref{t:conj-weighted projective}, and Corollary \ref{c:conj-marked surface}). We remark that the exchangeability property for $d$-compatibility degree is not true in general (cf. Section \ref{ss:counterexam}).
 
The structure of this paper is as follows:

In Section 2, we recall the fundamentals of cluster algebras. We also introduce some important properties of cluster algebras which will be used in later sections. In Section 3, we compare the $d$-vectors and the $f$-vectors. In \cite{cl2}, the $d$-compatibility degree was defined by some properties of $d$-vectors conjectured in \cite{fziv}.
To define the compatibility degree, we prove some properties of $f$-vectors which are analogous to properties of $d$-vectors (Theorem~\ref{ftheorem}). In Section 4, we introduce the classical compatibility degree and the compatibility degree, and we prove that the compatibility degree is a generalization of the classical one (Theorem~\ref{classical-fcorrespondence}). We compare properties of the compatibility degree with those of the $d$-compatibility degree (Proposition~\ref{pr:f-compatibility-symmetry}, Theorem \ref{fcompatibilityproperty}, Example \ref{counterexsym}). We also formulate the exchangeability conjecture for the compatibility degree (Theorem \ref{f-exchangeability1} and Conjecture \ref{f-exchangeability2}). By a description of $F$-matrices, we prove Conjecture \ref{f-exchangeability2} for cluster algebras of rank $2$ (Theorem \ref{f-exchangeability2rank2}).
In Section 5 and 6, we investigate the exchangeability property for certain important classes of cluster algebras. In particular, we establish Conjecture~\ref{f-exchangeability2} for acyclic cluster algebras of skew-symmetric type (Theorem~\ref{t:conj-acyclic}), cluster algebras arising from weighted projective lines (Theorem~\ref{t:conj-weighted projective}) and cluster algebras arising from marked surfaces (Corollary~\ref{c:conj-marked surface}).
Our approach relies on the existence of additive categorifications by $2$-Calabi-Yau triangulated categories for these cluster algebras. In Section~\ref{ss:cluster-tilting} and \ref{ss:cluster-character}, we recollect basic results on cluster-tilting theory in $2$-Calabi-Yau triangulated categories and the associated cluster character. Under mild conditions, we give a categorical interpretation of compatibility degree (Theorem~\ref{p:categorical-f-degree}). In Section~\ref{ss:exchangeable-property-2-CY}, we prove an exchangeability property for $2$-Calabi-Yau triangulated categories with cluster-tilting objects (Theorem~\ref{t:exchangeable-2-CY}), which is applied in Section~\ref{ss:ex-acylic}, \ref{ss:ex-wpl} and \ref{subsection:markedsurface} to deduce the exchangeability property for the corresponding cluster algebras.

\subsection*{Acknowledgement} The authors thank Tomoki Nakanishi and Osamu Iyama for their valuable comments. We thank Fang Li for informing us that he also obtained a counterexample for the exchangeability property of $d$-compatibility degree. We are grateful to the anonymous referee for significant comments and corrections. Fu's work was partially supported by NSF of China (No. 11971326). Gyoda's work was supported by JSPS KAKENHI Grant number JP20J12675.

\section{Preliminaries}
\subsection{Cluster algebras}
We start by recalling definitions of seed mutations and cluster patterns according to \cite{fziv}. A \emph{semifield} $\mathbb P$ is an abelian multiplicative group equipped with an addition $\oplus$ which is distributive over the multiplication. We particularly make use of the following two semifields.

Let $\mathbb Q_{\text{sf}}(u_1,\dots,u_{\ell})$ be the set of rational functions in $u_1,\dots,u_{\ell}$ which have subtraction-free expressions. Then, $\mathbb Q_{\text{sf}}(u_1,\dots,u_{\ell})$ is a semifield by the usual multiplication and addition. It is called the \emph{universal semifield} of $u_1,\dots,u_{\ell}$ (\cite{fziv}*{Definition 2.1}).

Let Trop$(u_1,\dots, u_\ell)$ be the abelian multiplicative group freely generated by the elements $u_1,\dots,u_\ell$. Then, $\text{Trop}(u_1,\dots,u_{\ell})$ is a semifield by the following addition: 
\begin{align}
\prod_{j=1}^\ell u_j^{a_j} \oplus \prod_{j=1}^{\ell} u_j^{b_j}=\prod_{j=1}^{\ell} u_j^{\min(a_j,b_j)}.
\end{align}
It is called the \emph{tropical semifield} of $u_1,\dots,u_\ell$ (\cite{fziv}*{Definition 2.2}).
For any semifield $\PP$ and $p_1, \dots, p_{\ell}\in\PP$, there exists a unique semifield homomorphism $\pi$ such that
\begin{align} \label{qsfuniv}
	\pi:\QQsf(y_1, \dots, y_{\ell}) &\longrightarrow \PP\\
	y_i &\longmapsto p_i. \nonumber
\end{align} 
For $F(y_1,\dots,y_\ell ) \in \QQsf(y_1, \dots, y_{\ell})$, we denote 
\begin{align}
	F|_{\PP}(p_1, \dots, p_{\ell}):=\pi(F(y_1, \dots, y_\ell)).
\end{align}
and it is called the \emph{evaluation} of $F$ at $p_1, \dots, p_{\ell}$.
We fix a positive integer $n$ and a semifield $\PP$. Let $\mathbb{ZP}$ be the group ring of $\mathbb{P}$ as a multiplicative group. Since $\mathbb{ZP}$ is a domain (\cite{fzi}*{Section 5}), its total quotient ring is a field $\mathbb{Q}(\mathbb P)$. Let $\mathcal{F}$ be the field of rational functions in $n$ indeterminates with coefficients in $\mathbb{Q}(\mathbb P)$. 

A \emph{labeled seed with coefficients in $\PP$} is a triplet $(\mathbf{x}, \mathbf{y}, B)$, where
\begin{itemize}
\item $\mathbf{x}=(x_1, \dots, x_n)$ is an $n$-tuple of elements of $\mathcal F$ forming a free generating set of $\mathcal F$.
\item $\mathbf{y}=(y_1, \dots, y_n)$ is an $n$-tuple of elements of $\mathbb{P}$.
\item $B=(b_{ij})$ is an $n \times n$ integer matrix which is \emph{skew-symmetrizable}, that is, there exists a positive integer diagonal matrix $S$ such that $SB$ is skew-symmetric. Also, we call $S$ a \emph{skew-symmetrizer} of $B$.
\end{itemize}

We say that $\xx$ is a \emph{cluster} and refer to $x_i,y_i$ and $B$ as the \emph{cluster variables}, the \emph{coefficients} and the \emph{exchange matrix}, respectively.

Throughout the paper, for an integer $b$, we use the notation $[b]_+=\max(b,0)$. We note that
\begin{align}\label{eq:b--b}
b=[b]_+-[-b]_+.
\end{align}
Let $(\mathbf{x}, \mathbf{y}, B)$ be a labeled seed with coefficients in $\PP$, and let $k \in\{1,\dots, n\}$. The \emph{seed mutation $\mu_k$ in direction $k$} transforms $(\mathbf{x}, \mathbf{y}, B)$ into another labeled seed $\mu_k(\mathbf{x}, \mathbf{y}, B)=(\mathbf{x'}, \mathbf{y'}, B')$ defined as follows:
\begin{itemize}
\item The entries of $B'=(b'_{ij})$ are given by 
\begin{align} \label{eq:matrix-mutation}
b'_{ij}=\begin{cases}-b_{ij} &\text{if $i=k$ or $j=k$,} \\ 
b_{ij}+\left[ b_{ik}\right] _{+}b_{kj}+b_{ik}\left[ -b_{kj}\right]_+ &\text{otherwise.}
\end{cases}
\end{align}
\item The coefficients $\mathbf{y'}=(y'_1, \dots, y'_n)$ are given by 
\begin{align}\label{eq:y-mutation}
y'_j=
\begin{cases}
y_{k}^{-1} &\text{if $j=k$,} \\ 
y_j y_k^{[b_{kj}]_+}(y_k \oplus 1)^{-b_{kj}} &\text{otherwise.}
\end{cases}
\end{align}
\item The cluster variables $\mathbf{x'}=(x'_1, \dots, x'_n)$ are given by
\begin{align}\label{eq:x-mutation}
x'_j=\begin{cases}\dfrac{y_k\mathop{\prod}\limits_{i=1}^{n} x_i^{[b_{ik}]_+}+\mathop{\prod}\limits_{i=1}^{n} x_i^{[-b_{ik}]_+}}{(y_k\oplus 1)x_k} &\text{if $j=k$,}\\
x_j &\text{otherwise.}
\end{cases}
\end{align}
\end{itemize}

Let $\mathbb{T}_n$ be the \emph{$n$-regular tree} whose edges are labeled by the numbers $1, \dots, n$ such that the $n$ edges emanating from each vertex have different labels. We write 
\begin{xy}(0,1)*+{t}="A",(10,1)*+{t'}="B",\ar@{-}^k"A";"B" \end{xy} 
to indicate that vertices $t,t'\in \mathbb{T}_n$ are joined by an edge labeled by $k$. We fix an arbitrary vertex $t_0\in \TT_n$, which is called the \emph{rooted vertex}.

A \emph{cluster pattern with coefficients in $\PP$} is an assignment of a labeled seed $\Sigma_t=(\mathbf{x}_t, \mathbf{y}_t,B_t)$ with coefficients in $\PP$ to every vertex $t\in \mathbb{T}_n$ such that the labeled seeds $\Sigma_t$ and $\Sigma_{t'}$ assigned to the endpoints of any edge 
\begin{xy}(0,1)*+{t}="A",(10,1)*+{t'}="B",\ar@{-}^k"A";"B" \end{xy} 
are obtained from each other by the seed mutation in direction $k$. The elements of $\Sigma_t$ are denoted as follows:
\begin{align} \label{den:seed_at_t}
\mathbf{x}_t=(x_{1;t},\dots,x_{n;t}),\ \mathbf{y}_t=(y_{1;t},\dots,y_{n;t}),\ B_t=(b_{ij;t}).
\end{align}
In particular, at $t_0$, we denote
\begin{align} \label{initialseed}
\mathbf{x}=\mathbf{x}_{t_0}=(x_1,\dots,x_n),\ \mathbf{y}=\mathbf{y}_{t_0}=(y_1,\dots,y_n),\ B=B_{t_0}=(b_{ij}).
\end{align}

\begin{definition}
A \emph{cluster algebra} $\Acal$ associated with a cluster pattern $v\mapsto \Sigma_v$ is the $\ZZ\PP$-subalgebra of $\Fcal$ generated by $\Xcal=\{x_{i;t}\}_{1\leq i\leq n, t\in \TT_n}$.
 \end{definition}

The degree $n$ of the regular tree $\TT_n$ is called the \emph{rank} of $\Acal$, and $\Fcal$ is the \emph{ambient field} of $\Acal$.
We also denote by $\mathcal{A}(B)$ a cluster algebra with the initial matrix $B$. 
\begin{example}\label{A2}
We give an example for mutations in the case of type $A_2$.
Let $n=2$, and we consider a tree $\TT_2$ whose edges are labeled as follows:
\begin{align}\label{A2tree}
\begin{xy}
(-10,0)*+{\dots}="a",(0,0)*+{t_0}="A",(10,0)*+{t_1}="B",(20,0)*+{t_2}="C", (30,0)*+{t_3}="D",(40,0)*+{t_4}="E",(50,0)*+{t_5}="F", (60,0)*+{\dots}="f"
\ar@{-}^{1}"a";"A"
\ar@{-}^{2}"A";"B"
\ar@{-}^{1}"B";"C"
\ar@{-}^{2}"C";"D" 
\ar@{-}^{1}"D";"E"
\ar@{-}^{2}"E";"F" 
\ar@{-}^{1}"F";"f" 
\end{xy}.
\end{align}
We set $B=\begin{bmatrix}
 0 & 1 \\
 -1 & 0
\end{bmatrix}
$ as the initial exchange matrix at $t_0$.
Then, coefficients and cluster variables are given by Table \ref{A2seed} \cite{fziv}*{Example 2.10}.
\begin{table}[ht]
\begin{equation*}
\begin{array}{|c|cc|cc|}
\hline
&&&&\\[-4mm]
t& \hspace{25mm}\yy_t &&& \xx_t \hspace{30mm}\\
\hline
&&&&\\[-3mm]
0 &y_1 & y_2& x_1& x_2 \\[1mm]
\hline
&&&&\\[-3mm]
1& y_1(y_2\oplus 1)& \dfrac{1}{y_2} & x_1& \dfrac{x_1y_2+1}{(y_2\oplus 1)x_2} \\[3mm]
\hline
&&&&\\[-3mm]
2& \dfrac{1}{y_1(y_2\oplus 1)} & \dfrac{y_1y_2\oplus y_1\oplus 1}{y_2} & \dfrac{x_1y_1y_2 + y_1+ x_2}{(y_1y_2\oplus y_1\oplus 1)x_1x_2} & \dfrac{x_1y_2+1}{(y_2\oplus 1)x_2} \\[3mm]
\hline
&&&&\\[-3mm]
3& \dfrac{y_1\oplus1}{y_1y_2} & \dfrac{y_2}{y_1y_2\oplus y_1\oplus 1} & \dfrac{x_1y_1y_2+y_1+x_2}{(y_1y_2\oplus y_1\oplus 1)x_1x_2} & \dfrac{y_1+x_2}{x_1(y_1\oplus 1)} \\[3mm]
\hline
&&&&\\[-2mm]
4& \dfrac{y_1y_2}{y_1\oplus 1} &\dfrac{1}{y_1} & x_2 & \dfrac{y_1+x_2}{x_1(y_1\oplus 1)} \\[3mm]
\hline
&&&&\\[-2mm]
5& y_2 & y_1 & x_2 & x_1\\[1mm]
\hline
\end{array}
\end{equation*}
\caption{Coefficients and cluster variables in type~$A_2$\label{A2seed}}
\end{table}

Therefore, we have
\begin{align*}
\Acal(B)=\ZZ\PP\left[x_1,x_2,\dfrac{x_1y_2+1}{(y_2\oplus 1)x_2},\dfrac{x_1y_1y_2+y_1+x_2}{(y_1y_2\oplus y_1\oplus 1)x_1x_2}, \dfrac{y_1+x_2}{x_1(y_1\oplus 1)}\right].
\end{align*}
\end{example}

Next, to define the class of cluster algebras of finite type, we define the non-labeled seeds according to \cite{fziv}. For a cluster pattern $v\mapsto \Sigma_v$, we introduce the following equivalence relations of labeled seeds: We say that \begin{align*}
\Sigma_t=(\xx_t, \yy_t, B_t),\quad \xx_t=(x_{1;t,}\dots,x_{n;t}),\quad
\yy_t=(y_{1;t},\dots,y_{n;t}),\quad B_t=(b_{ij;t})
\end{align*}
and
\begin{align*}
\Sigma_s=(\xx_{s}, \yy_{s}, B_{s}),\quad \xx_s=(x_{1;s},\dots,x_{n;s}),\quad
\yy_s=(y_{1;s},\dots,y_{n;s}),\quad B_s=(b_{ij;s})
\end{align*}
are equivalent if there exists a
permutation~$\sigma$ of indices~$1, \dots, n$ such that
\begin{align*}
x_{i;s} = x_{\sigma(i);t}, \quad y_{j;s} = y_{\sigma(j);t}, \quad
b_{ij;s} = b_{\sigma(i), \sigma(j);t}
\end{align*}
for all~$i$ and~$j$.
We denote by $[\Sigma]$ the equivalence classes represented by a labeled seed
$\Sigma$ and call it the \emph{non-labeled seed}. We define the \emph{non-labeled clusters} (resp., \emph{non-labeled coefficients}) as clusters (resp., coefficients) of non-labeled seeds.
\begin{definition}
The \emph{exchange graph} of a cluster algebra is the regular connected graph whose vertices are the non-labeled seeds of the cluster pattern and whose edges connect the non-labeled seeds related by a single mutation.
\end{definition}
Using the exchange graph, we define cluster algebras of finite type.
\begin{definition}
A cluster algebra $\Acal$ is of \emph{finite type} if the exchange graph of $\Acal$ is a finite graph.
\end{definition}
\subsection{$d$-vectors, $c$-vectors, $g$-vectors, and $f$-vectors}
In this subsection, we define the $d$-vectors, $c$-vectors, $g$-vectors, and the $f$-vectors. First, we define the $d$-vectors according to \cites{fzii,fziv}. 
Let $\Acal$ be a cluster algebra.
By the \emph{Laurent phenomenon} \cite{fziv}*{Theorem 3.5}, every cluster variable
$x_{j;t} \in \Acal$ can be uniquely written as
\begin{align}\label{eq:Laurent-normal-form}
x_{j;t} = \frac{N_{j;t}(x_1, \dots, x_n)}{x_1^{d_{1j;t}} \cdots x_n^{d_{nj;t}}},\quad d_{kj;t}\in \ZZ,
\end{align}
where $N_{j;t}(x_1, \dots, x_n)$ is a polynomial with coefficients in~$\ZZ \PP$
which is not divisible by any initial cluster variable~$x_j\in\xx$.
\begin{definition}
We define the \emph{$d$-vector} $\dd_{j;t}$ as the degree vector of $x_{j;t}$, that is, 
\begin{align}
\dd_{j;t}^{B;t_0}=\dd_{j;t}=\begin{bmatrix}d_{1j;t}\\ \vdots \\ d_{nj;t} \end{bmatrix},
\end{align}
in \eqref{eq:Laurent-normal-form}. We define the \emph{$D$-matrix} $D_t^{B;t_0}$ as 
\begin{align}
D_t^{B;t_0}:=(\dd_{1;t},\dots,\dd_{n;t}).
\end{align}
\end{definition}
\begin{remark}
We remark that
$\dd_{j;t}$ is independent of the choice of the
coefficient system (see \cite{fziv}*{Section 7}). Thus, we can also regard $d$-vectors as vectors associated with vertices of $\TT_n$. They are also given by the following recursion: 
For any $j\in\{1,\dots,n\}$,
\begin{align*}
\dd_{j;t_0}=-\mathbf{e}_j,
\end{align*}
and for any \begin{xy}(0,1)*+{t}="A",(10,1)*+{t'}="B",\ar@{-}^k"A";"B" \end{xy}, 
\begin{align}\label{dvectorrecursion}
\mathbf{d}_{j;{t'}}&=\begin{cases}
\mathbf{d}_{j;t} \ \ & \text{if } j\neq k;\\
-\mathbf{d}_{k;t}+\max \left(\mathop{\sum}\limits_{i=1}^n[b_{ik;t}]_+\mathbf{d}_{i;t},\ +\mathop{\sum}\limits_{i=1}^n[-b_{ik;t}]_+\mathbf{d}_{i;t}\right)\ \ &\text{if } j=k,
\end{cases}
\end{align}
where $\mathbf{e}_j$ is the $j$th canonical basis.
\end{remark}

Next, we define the $c$-vectors, the $g$-vectors and the $f$-vectors according to \cites{fziv, fk, fg}. We introduce the principal coefficients to define them.
\begin{definition}
We say that a cluster pattern $v\mapsto \Sigma_v$ or a cluster algebra $\Acal$ of rank $n$ has \emph{principal coefficients} at the rooted vertex $t_0$ if $\mathbb{P}=\text{Trop}(y_1,\dots,y_n)$ and $\mathbf{y}_{t_0}=(y_1,\dots,y_n)$. In this case, we denote $\Acal=\Acal_\bullet(B)$.
\end{definition}
First, we define the $c$-vectors. For $\bb=(b_1,\dots,b_n)^{\top}$, we use the notation $[\bb]_+=([b_1]_+,\dots,[b_n]_+)^{\top}$, where $\top$ means transposition. 
\begin{definition}
Let $\Acal_{\bullet}(B)$ be a cluster algebra with principal coefficients at $t_0$. We define the \emph{$c$-vector} $\cc_{j;t}$ as the degree of $y_i$ in $y_{j;t}$, that is, if $y_{j;t}=y_1^{c_{1j;t}}\cdots y_n^{c_{nj;t}}$, then
\begin{align}
\cc_{j;t}^{B;t_0}=\cc_{j;t}=\begin{bmatrix}c_{1j;t}\\ \vdots \\ c_{nj;t} \end{bmatrix}.
\end{align}
We define the \emph{$C$-matrix} $C_t^{B;t_0}$ as 
\begin{align}
C_t^{B;t_0}:=(\cc_{1;t},\dots,\cc_{n;t}).
\end{align}
\end{definition}
The $c$-vectors are the same as those defined by the following recursion: For any $j \in\{1,\dots,n\}$,
\begin{align*}
\cc_{j;t_0}=\ee_j\quad \text{(canonical basis)},
\end{align*}
and for any \begin{xy}(0,1)*+{t}="A",(10,1)*+{t'}="B",\ar@{-}^k"A";"B" \end{xy}, 
\begin{align*}
\cc_{j;t'} =
\begin{cases}
-\cc_{j;t} & \text{if $j=k$;} \\[.05in]
\cc_{j;t} + [b_{kj;t}]_+ \ \cc_{k;t} +b_{kj;t} [-\cc_{k;t}]_+
 & \text{if $j\neq k$}
 \end{cases}
\end{align*}
(see \cite{fziv}). Since the recursion formula only depends on exchange matrices,
we can regard $c$-vectors as vectors associated with vertices of $\TT_n$. 

{In this way, we remark that it is possible to define $c$-vectors for a cluster algebra that does not have principal coefficients.} 

The $C$-matrices have the following property:

\begin{theorem}[{\cite{GHKK}*{Corollary 5.5}}]
\label{thm:signs-ci}
For any initial exchange matrix $B$ and every column of the $C$-matrix $C_t^{B;t_0}\ (t\in \TT_n)$, that is, any $c$-vectors, its entries are either all nonnegative, or all nonpositive, and not all zero. 
\end{theorem}

We call this property the \emph{sign-coherence of the $C$-matrices}.

Next, we define the $g$-vectors. We can regard cluster variables in cluster algebras with principal coefficients as homogeneous Laurent polynomials:

\begin{theorem}[\cite{fziv}*{Proposition 6.1}]
Let $\Acal_{\bullet}(B)$ be a cluster algebra with principal coefficients at $t_0$. Each cluster variable $x_{i;t}$ is a homogeneous Laurent polynomial in $x_1,\dots,x_n,y_1,\dots,y_n$ by the following $\ZZ^n$-grading:
\begin{align}\label{grading}
\deg x_{i}=\ee_i,\quad \deg y_{i}=-\mathbf{b}_i,
\end{align}
where $\ee_i$ is the $i$th canonical basis of $\ZZ^n$ and $\mathbf{b}_i$ is the $i$th column vector of $B$.
\end{theorem}
We denote by $\begin{bmatrix}g_{1j;t}\\ \vdots \\ g_{nj;t} \end{bmatrix}$ the $\ZZ^n$-grading of $x_{j;t}$. 
\begin{definition}
Let $\Acal_{\bullet}(B)$ be a cluster algebra with principal coefficients at $t_0$. We define the \emph{$g$-vector} $\gg_{j;t}$ as the degree of the homogeneous Laurent polynomial $x_{j;t}$ by the $\ZZ^n$-grading \eqref{grading}, that is, 
\begin{align}
\gg_{j;t}^{B;t_0}=\gg_{j;t}=\begin{bmatrix}g_{1j;t}\\ \vdots \\ g_{nj;t} \end{bmatrix}.
\end{align}
We define the \emph{$G$-matrix} $G_t^{B;t_0}$ as 
\begin{align}
G_t^{B;t_0}:=(\gg_{1;t},\dots,\gg_{n;t}).
\end{align}
\end{definition}
The $g$-vectors are the same as those defined by the following recursion: For any $j\in\{1,\dots,n\}$,
\begin{align*}
\gg_{j;t_0}=\ee_j\quad \text{(canonical basis)},
\end{align*}
and for any \begin{xy}(0,1)*+{t}="A",(10,1)*+{t'}="B",\ar@{-}^k"A";"B" \end{xy}, 
\begin{align}\label{g-recursion}
\gg_{j;{t'}}&=\begin{cases}
\gg_{j;t} \ \ & \text{if } j\neq k;\\
-\gg_{k;t}+\mathop\sum\limits_{i=1}^{n}[b_{ik;t}]_+\gg_{i;t}-\mathop\sum\limits_{i=1}^{n}[c_{ik;t}]_+\mathbf{b}_j &\text{if } j=k
\end{cases}
\end{align}
(see \cite{fziv}). Since the recursion formula only depend on exchange matrices,
we can regard $g$-vectors as vectors associated with vertices of $\TT_n$. 


{In this way, we remark that it is possible to define $g$-vectors for a cluster algebra that does not have principal coefficients.} 

The $G$-matrices have the following property, which is the dual one of the sign-coherence of the $C$-matrices: 

\begin{theorem}[{\cite{GHKK}*{Corollary 5.11}}]
\label{thm:signs-g}
For any initial exchange matrix $B$ and every row of the $G$-matrix $G_t^{B;t_0}\ (t\in \TT_n)$, its entries are either all nonnegative, or all nonpositive, and not all zero. 
\end{theorem}

We call this property the \emph{sign-coherence of the $G$-matrices}.


Next, we define the $F$-polynomials and the $f$-vectors.
\begin{definition}
Let $\Acal_{\bullet}(B)$ be a cluster algebra with principal coefficients at $t_0$. we define the \emph{$F$-polynomial} $F^{B;t_0}_{i;t}(\yy)$ as 
\begin{align}
F^{B;t_0}_{i;t}(\yy)=x_{i;t}(x_1,\dots,x_n;y_1,\dots,y_n)|_{x_1=\cdots=x_n=1},
\end{align}
where $x_{i;t}(x_1,\dots,x_n;y_1,\dots,y_n)$ means the expression of $x_{i;t}$ by $x_1,\dots,x_n,y_1,\dots,y_n$. 
\end{definition}
The following property of $F$-polynomials is a consequence of Theorem \ref{thm:signs-ci} and \cite{fziv}*{Proposition 5.6}.

\begin{proposition}\label{fconstant1}
Every polynomial $F_{\ell;t}^{B;t_0}(\yy)$ has constant term~$1$.
\end{proposition}

Using the $F$-polynomials, we define the $f$-vectors. 

\begin{definition}\label{def:f-vector}
Let $\Acal_{\bullet}(B)$ be a cluster algebra with principal coefficients at $t_0$. We denote by $f_{ij;t}$ the maximal degree of $y_i$ in $F_{j;t}^{B;t_0}(\yy)$. Then, we define the \emph{$f$-vector} $\ff_{j;t}$ as

\begin{align}
\ff_{j;t}^{B;t_0}=\ff_{j;t}=\begin{bmatrix}f_{1j;t}\\ \vdots \\ f_{nj;t} \end{bmatrix}.
\end{align}
We define the \emph{$F$-matrix} $F_t^{B;t_0}$ as 

\begin{align}
F_t^{B;t_0}:=(\ff_{1;t},\dots,\ff_{n;t}).
\end{align}
\end{definition}
The $F$-polynomials are the same as those defined by the following recursion: 
For any $j\in\{1,\dots,n\}$,
\begin{align*}
F_{j;t_0}^{B;t_0}=1
\end{align*}
and for any \begin{xy}(0,1)*+{t}="A",(10,1)*+{t'}="B",\ar@{-}^k"A";"B" \end{xy}, 
\begin{align}\label{Fpoly-recursion}
F_{j;t'}^{B;t_0}(\yy)=\begin{cases}
F_{j;t}^{B;t_0}(\yy) &\text{if $j\neq k$};\vspace{2mm}\\
\dfrac{\mathop{\prod}\limits_{i=1}^{n}y_i^{[c_{ik}]_+}\mathop{\prod}\limits_{i=1}^{n} \left(F_{i;t}^{B;t_0}(\yy)\right)^{[b_{ik}]_+}+\mathop{\prod}\limits_{i=1}^{n}y_i^{[-c_{ik}]_+}\mathop{\prod}\limits_{i=1}^{n} \left(F_{i;t}^{B;t_0}(\yy)\right)^{[-b_{ik}]_+}}{F_{k;t}^{B;t_0}(\yy)} &\text{if $j=k$}.
\end{cases}
\end{align}
Also, the $f$-vectors are the same as those defined by the following recursion: For any $j\in\{1,\dots,n\}$,
\begin{align*}
\ff_{j;t_0}=\mathbf{0},
\end{align*}
and for any \begin{xy}(0,1)*+{t}="A",(10,1)*+{t'}="B",\ar@{-}^k"A";"B" \end{xy}, 
\begin{align}\label{f-recursion}
\mathbf{f}_{j;{t'}}&=\begin{cases}
\mathbf{f}_{j;t} \ \ & \text{if } j\neq k;\\
-\mathbf{f}_{k;t}+\max \left([\mathbf{c}_{k;t}]_++\mathop{\sum}\limits_{i=1}^n[b_{ik;t}]_+\mathbf{f}_{i;t},\ [-\mathbf{c}_{k;t}]_++\mathop{\sum}\limits_{i=1}^n[-b_{ik;t}]_+\mathbf{f}_{i;t}\right)\ \ &\text{if } j=k
\end{cases}
\end{align}
(about how to get the recursion of the $f$-vectors, see \cite{fg}). Comparing the $d$-vector's and $f$-vector's reccursions, \eqref{dvectorrecursion} and \eqref{f-recursion}, we can see that they are very similar. The similarity between the $d$-vectors and the $f$-vectors will be discussed in Section 3.
Just like $c$-vectors and $g$-vectors, we can regard $F$-polynomials and $f$-vectors as polynomials and vectors associated with vertices of $\TT_n$. 


{In this way, we remark that it is possible to define $F$-polynomials and $f$-vectors for a cluster algebra that does not have principal coefficients.} 

The $c$-vector, the $g$-vector, the $F$-polynomials and the exchange matrices can restore the cluster variables and the coefficients:

\begin{proposition}[\cite{fziv}*{Proposition 3.13, Corollary 6.3}]
\label{pr:separation}
Let $\{\Sigma_t\}_{t\in\TT_n}$ be a cluster pattern with coefficients in $\PP$ with the initial seed \eqref{initialseed}. 
Then, for any $t \in \TT_n$ and $j \in \{1, \dots, n\}$, we have
\begin{align}
\label{eq:xjt=F/F}
x_{j;t} &= \left( \prod_{k=1}^n x_k^{g^{B; t_0}_{kj;t}} \right) \frac{F_{j;t}^{B;t_0}|_\Fcal(\hat y_1, \dots, \hat y_n)}{F_{j;t}^{B;t_0}|_\PP (y_1, \dots, y_n)}, \\
\label{eq:Y-F}
y_{j;t}&=\prod_{k=1}^n y_{k}^{c^{B; t_0}_{kj;t}}
\prod_{k=1}^n (F_{k;t}^{B;t_0}|_{\mathbb{P}}(y_{1}, \dots, y_{n}))^{b_{kj;t}},
\end{align}
where 
\begin{align}
\hat y_i = y_i \mathop{\prod}\limits_{j=1}^n x_j^{b_{ji}},
\end{align}
and $g_{ij;t}^{B;t_0}$ and $c_{ij;t}^{B;t_0}$ are the $(i,j)$ entry of $G_t^{B;t_0}$ and $C_t^{B;t_0}$, respectively. Also, the rational function $F_{j;t}^{B;t_0}|_\Fcal(\hat y_1, \dots, \hat y_n)$ is the element of $\Fcal$ obtained by substituting $\hat{y}_i$ for $y_i$ in $F_{j;t}^{B;t_0}(y_1, \dots, y_n)$.
\end{proposition}
We call \eqref{eq:xjt=F/F} and \eqref{eq:Y-F} the \emph{separation formulas}.

\begin{example}\label{ex:prinA2} 
Let $\Acal(B)$ be the cluster algebra given in Example \ref{A2}. In particular, we take one with principal coefficients at $t_0$, that is, we consider $\Acal_{\bullet}(B)$. 
Then, clusters and coefficient tuples are given by Table \ref{prinA2} and the $F$-polynomials and the $F,D,C,G$-matrices are given by Table \ref{A2FDCG}.
\begin{table}[ht]
\begin{equation*}
\begin{array}{|c|cc|cc|}
\hline
&&&&\\[-4mm]
t& \hspace{12mm}\yy_t &&& \xx_t \hspace{30mm}\\
\hline
&&&&\\[-3mm]
0 &y_1 & y_2& x_1& x_2 \\[1mm]
\hline
&&&&\\[-3mm]
1& y_1& \dfrac{1}{y_2} & x_1& \dfrac{x_1y_2+1}{x_2} \\[3mm]
\hline
&&&&\\[-3mm]
2& \dfrac{1}{y_1} & \dfrac{1}{y_2} & \dfrac{x_1y_1y_2 + y_1+ x_2}{x_1x_2} & \dfrac{x_1y_2+1}{x_2} \\[3mm]
\hline
&&&&\\[-3mm]
3& \dfrac{1}{y_1y_2} & y_2 & \dfrac{x_1y_1y_2+y_1+x_2}{x_1x_2} & \dfrac{y_1+x_2}{x_1} \\[3mm]
\hline
&&&&\\[-2mm]
4& {y_1y_2} &\dfrac{1}{y_1} & x_2 & \dfrac{y_1+x_2}{x_1} \\[3mm]
\hline
&&&&\\[-2mm]
5& y_2 & y_1 & x_2 & x_1\\[1mm]
\hline
\end{array}
\end{equation*}
\caption{Coefficients and cluster variables in type~$A_2$\label{prinA2}}
\end{table}
\begin{table}[ht]
\begin{equation*}
\begin{array}{|c|cc|c|c|c|c|}
\hline
&&&&&&\\[-4mm]
t& \hspace{0mm}F^{B; t_0}_{1;t}(\yy)&F^{B; t_0}_{2;t}(\yy) & F^{B; t_0}_t \hspace{0mm} & D^{B;t_0}_t & C^{B; t_0}_t & G^{B; t_0}_t\\
\hline
&&&&&&\\[-3mm]
0& 1&1 
& \begin{bmatrix}
 0 & 0 \\
 0 & 0
\end{bmatrix}
&\begin{bmatrix}
 -1 & 0 \\
 0 & -1
\end{bmatrix}
&\begin{bmatrix}
 1 & 0 \\
 0 & 1
\end{bmatrix}
&\begin{bmatrix}
 1 & 0 \\
 0 & 1
\end{bmatrix}
 \\[4mm]
\hline
&&&&&&\\[-3mm]
1& 1&y_2+1
&\begin{bmatrix}
 0 & 0 \\
 0 & 1
\end{bmatrix}
&\begin{bmatrix}
 -1& 0 \\
 0 & 1
\end{bmatrix}
&\begin{bmatrix}
 1 & 0 \\
 0 & -1
\end{bmatrix}
&\begin{bmatrix}
 1 & 0 \\
 0 & -1
\end{bmatrix}
\\[4mm]
\hline
&&&&&&\\[-3mm]
2& y_1y_2+y_1+1&y_2+1
&\begin{bmatrix}
 1 & 0 \\
 1 & 1
\end{bmatrix}
&\begin{bmatrix}
 1 & 0 \\
 1 & 1
\end{bmatrix}
&\begin{bmatrix}
 -1 & 0 \\
 0 & -1
\end{bmatrix}
&\begin{bmatrix}
 -1 & 0 \\
 0 & -1
\end{bmatrix}
\\[4mm]
\hline
&&&&&&\\[-3mm]
3& y_1y_2+y_1+1&y_1+1
&\begin{bmatrix}
 1 & 1 \\
 1 & 0
\end{bmatrix}
&\begin{bmatrix}
 1 & 1 \\
 1& 0
\end{bmatrix}
&\begin{bmatrix}
 -1 & 0 \\
 -1 & 1
\end{bmatrix}
&\begin{bmatrix}
 -1 & -1 \\
 0 & 1
\end{bmatrix}
\\[4mm]
\hline
&&&&&&\\[-3mm]
4& 1&y_1+1
&\begin{bmatrix}
 0 & 1 \\
 0& 0
\end{bmatrix}
&\begin{bmatrix}
 0 & 1 \\
 -1 & 0
\end{bmatrix} 
&\begin{bmatrix}
 1 & -1 \\
 1 & 0
\end{bmatrix}
&\begin{bmatrix}
 0 & -1 \\
 1 & 1
\end{bmatrix}
\\[4mm]
\hline
&&&&&&\\[-3mm]
5& 1&1
&\begin{bmatrix}
 0 & 0 \\
 0 & 0
\end{bmatrix}
&\begin{bmatrix}
 0 & -1 \\
 -1 & 0
\end{bmatrix}
&\begin{bmatrix}
 0 & 1 \\
 1 & 0
\end{bmatrix}
&\begin{bmatrix}
 0 & 1 \\
 1 & 0
\end{bmatrix}
\\[4mm]
\hline
\end{array}
\end{equation*}
\caption{$F$-polynomials, $F,D,C,G$-matrices in type~$A_2$\label{A2FDCG}}
\end{table}
\end{example}
Proposition \ref{pr:separation} implies that if $x_{i;t}=x_{j;t'}$ in $\Acal_{\bullet}(B)$, then for any cluster algebra $\Acal(B)$ having the same exchange matrix as $\Acal_{\bullet}(B)$ at $t_0$, the $i$th cluster variable associated with $t\in \TT_n$ is same as the $j$th one associated with $t'\in \TT_n$. More generally, the following fact is known:
\begin{proposition}[\cite{cl2}*{Prposition 6.1 (i)}]\label{independentP}
Let $\Acal_1(B)$ and $\Acal_2(B)$ be cluster algebras having the same exchange matrix at $t_0$. Let $\PP_1$ and $\PP_2$ be coefficients of $\Acal_1(B)$ and $\Acal_2(B)$, respectively. Denoted by $(\xx_t(k),\yy_t(k), B_t(k))$, the seed of $\Acal_k(B)$ at $t\in\TT_n$, $k= 1,2$. Then, $x_{i;t}(1) =x_{j;t'}(1)$ if and only if $x_{i;t}(2) =x_{j;t'}(2)$, where $t, t'\in\TT_n$ and $i, j\in \{1,2,\cdots, n\}$.
\end{proposition}
\begin{remark}
By recursions \eqref{eq:x-mutation}, \eqref{g-recursion} ,\eqref{Fpoly-recursion}, \eqref{f-recursion}, and Proposition \ref{independentP}, for any cluster algebra $\Acal(B)$, if $x_{i;t}=x_{j;t'}$ then we have $\gg_{i;t}=\gg_{j;t'},\ F_{i;t}^{B;t_0}(\yy)=F_{j;t'}^{B;t_0}(\yy)$, and $\ff_{i;t}=\ff_{j;t'}$ (we remark that these vectors and polynomials depend only on $\TT_n$ and index $i\in\{1,\dots,n\}$). Therefore, we can say that $\gg_{i;t},\ F_{i;t}^{B;t_0}(\yy)$, and $\ff_{i;t}$ are the $g$-vector, the $F$-polynomial and the $f$-vector associated with $x_{i;t}$, respectively.
\end{remark}
Let us introduce cluster complexes which were defined in \cite{fzii}.
\begin{definition}
Let $\Acal(B)$ be a cluster algebra. We define the cluster complex $\Delta(\Acal(B))$ as the simplicial complex whose simplexes are subsets of cluster variables which are contained in a cluster. 
\end{definition}

\begin{example}\label{clustercomplexofA2}
We consider the cluster algebra in Example \ref{ex:prinA2}. We give a cluster complex corresponding to this cluster algebra in Figure \ref{A2complex}.

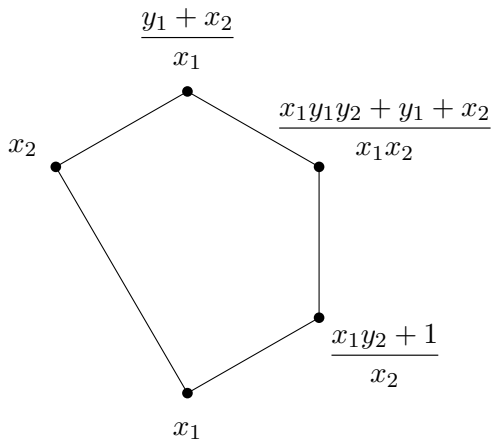
\begin{figure}[ht]
\caption{Cluster complex of type $A_2$}\label{A2complex}
\[
\begin{tikzpicture}
 \coordinate (0) at (0,0) ;
 \coordinate (1) at (30:2) ;
 \coordinate (2) at (90:2);
 \coordinate (3) at (150:2);
 \coordinate (4) at (270:2);
 \coordinate (5) at (330:2);
 \node (6) at (30:3) {$\dfrac{x_1y_1y_2 + y_1+ x_2}{x_1x_2}$};
 \node (7) at (90:2.7){$\dfrac{y_1+x_2}{x_1}$};
 \node (8) at (150:2.5){$x_2$};
 \node (9) at (270:2.5){$x_1$};
 \node (10) at (330:3){$\dfrac{x_1y_2+1}{x_2}$};
 \fill(1) circle (0.7mm);
 \fill(2) circle (0.7mm);
 \fill(3) circle (0.7mm);
 \fill(4) circle (0.7mm);
 \fill(5) circle (0.7mm);
 \draw(1) to (2);
 \draw(2) to (3);
 \draw(3) to (4);
 \draw(4) to (5);
 \draw(5) to (1);
\end{tikzpicture}
\]
\end{figure}
\end{example}

By Proposition \ref{independentP}, a cluster complex depends only on $B$ and does not depend on coefficients $\PP$.
\subsection{Scattering diagrams} The scattering diagrams were introduced in \cite{GHKK} to study the canonical basis of cluster algebras. In this paper, we give only a few definitions and properties used in proofs along \cites{mur,cl2,rea20, nak21}. 
For any $\vv=(v_1,\dots,v_n)^{\top}$ and $\mathbf{z}=(z_1,\dots,z_n)$, we abbreviate $z_1^{v_1}\cdots z_n^{v_n}$ as $\mathbf{z}^{\vv}$.

\begin{definition}\label{scatteringdiagram}
We fix a skew-symmetrizable matrix $B$ of order $n$ and take a skew-symmetrizer $S=\text{diag}(s_1,\dots,s_n)$ of $B$ so that $s_1,\dots,s_n$ are relatively prime. We call the pair $(\vv, W)$ a \emph{wall} associated with $B$, where $\vv\in \ZZ^n_{\geq 0}$ is a non-zero whose entries are relatively prime to each other, and $W$ is a convex cone spanning $\vv^{\perp}:=\{\mm\in \mathbb{R}^n \mid \vv^{\top}S\mm=0\}$.
\end{definition}

The open half space $\{\mm\in \mathbb{R}^n \mid \vv^{\top}S\mm>0\}$ is the \emph{green side} of $W$ (or a hyperplain including $W$), and $\{\mm\in \mathbb{R}^n \mid \vv^{\top}S\mm<0\}$ is the \emph{red side} of $W$. We consider $R=\QQ[x_1^{\pm1},\dots,x_n^{\pm1}][[y_1,\dots,y_n]]$, that is, formal power series in the variables $y_1,\dots,y_n$ with coefficients in $\QQ[x_1^{\pm1},\dots,x_n^{\pm1}]$.
For $\vv=(v_1,\dots,v_n)^\top \in \ZZ_{\geq 0}^n$, we define the \emph{formal elementary transformation} $E_\vv \in \text{Aut}(R)$ as 
\begin{align*}
 E_\vv(\xx^\ww)=\xx^\ww(1+\hat{\yy}^\vv)^{\frac{\vv^{\top}S\ww}{\gcd(s_1v_1,\dots,s_nv_n)}}, \quad E_\vv(\yy^{\ww'})=\yy^{\ww'}
\end{align*}
with the inverse 
\begin{align*}
 E_\vv^{-1}(\xx^\ww)=\xx^\ww(1+\hat{\yy}^\vv)^{-\frac{\vv^{\top}S\ww}{\gcd(s_1v_1,\dots,s_nv_n)}}, \quad E_\vv^{-1}(\yy^{\ww'})=\yy^{\ww'},
\end{align*}
where $\hat{\yy}=(\hat{y_1},\dots,\hat{y_n})$ and $\hat y_i$ is given by 
\begin{gather*}
\hat y_i = y_i \mathop{\prod}\limits_{j=1}^n x_j^{b_{ji}}.
\end{gather*}

Let $\mathfrak D(B)$ be a set of walls associated with $B$ and $I$ a monomial ideal of $\QQ[[y_1,\dots,y_n]]$. The \emph{reduction} $\mathfrak{D}(B)/I$ of $\mathfrak{D}(B)$ is obtained from $\mathfrak{D}(B)$ by deleting all walls of the form $(\vv, W)$ with $y^\vv \in I$. For a non-negative integer $k$, denote by $I_k$ the ideal of $\QQ[[y_1,\dots,y_n]]$ generated by all the monomials of $\yy$ with total degree $\geq k$. 

\begin{definition}
We say that a set $\mathfrak{D}(B)$ of walls associated with $B$ is a \emph{scattering diagram} of $B$ if it satisfies the following finite condition:

\begin{itemize}
  \item For each $k\in\ZZ_{\geq 0}$, $\mathfrak{D}(B)$ has finitely many walls in $\mathfrak{D}(B)/I_k$.
\end{itemize}
In this case, each connected compartment of $\mathbb{R}^n-\mathfrak{D}(B)$ is a \emph{chamber} of $\mathfrak{D}(B)$. 
\end{definition}
\begin{definition}\label{finitetranverse}
Let $\mathfrak{D}(B)$ be a scattering diagram in $\mathbb{R}^n$. We say that a continued path $\rho\colon [0,1]\rightarrow \mathbb{R}^n$ is \emph{generic} for $\mathfrak{D}(B)$ if $\rho$ satisfies the following conditions:
\begin{itemize}
\item $\rho(0)$ and $\rho(1)$ is not in any wall $W$ of $\mathfrak{D}(B)$.
\item The image of $\rho$ crosses each wall $W$ of $\mathfrak{D}(B)$ transversely.
\item The image of $\rho$ does not cross the boundaries of walls or intersection of walls spanning two distinct hyperplanes. 
\end{itemize}
\end{definition}

Let $\rho$ be a generic path of a scattering diagram $\mathfrak{D}(B)$, and we assume that $\rho$ crosses walls of $\mathfrak{D}(B)/I_k$ in the following order:
\begin{align*}
 (\vv_{1,k},W_{1,k}),\dots,(\vv_{s,k},W_{s,k}).
\end{align*}
We set 
\begin{align*}
 E^k_\rho=E_{\vv_{s,k}}^{\varepsilon_{s,k}}\circ\cdots\circ E_{\vv_{1,k}}^{\varepsilon_{1,k}}\in \text{Aut}(R),
\end{align*}
where $\varepsilon_{i,k} =1$ (resp., $\varepsilon_{i,k} =-1$) if $\rho$ crosses $W_{i,k}$ from its green side to its red side (resp., from its red side to its green side). We define the \emph{path-ordered product} of $\rho$ as 
\begin{align*}
 E_\rho=\lim_{k\to\infty}E^k_\rho \in \text{Aut}(R).
\end{align*}

A scattering diagram $\mathfrak{D}(B)$ is \emph{consistent} if for any generic path $\rho$ whose starting point coincides with terminal point, $E_\rho=1$ holds. We say that two scattering diagrams $\mathfrak{D}_1(B)$ and $\mathfrak{D}_2(B)$ for $B$ are \emph{equivalent} if any generic path in both diagrams determines the same path-ordered product.

\begin{lemma}[{\cite{GHKK}*{Theorem 1.12}}]\label{consistentdiagram}
Let $B$ be a skew-symmetrizable matrix of order $n$. There exists a unique consistent scattering diagram $\mathfrak{D}_0(B)$, {up to equivalence}, satisfying the following conditions:
\begin{itemize}
\item For any $i\in\{1,\dots,n\}$, $(\ee_i,\ee_i^{\perp})$ are walls of $\mathfrak{D}_0(B)$.
\item For any other walls $(\vv,W)$ of $\mathfrak{D}_0(B)$, $B\vv\notin W$ holds.
\end{itemize}
\end{lemma}
We note that since $\vv\in \ZZ_{\geq 0}$, we have $W\cap(\mathbb{R}_{>0})^n=W\cap(\mathbb{R}_{<0})^n=\emptyset$. Therefore, $(\mathbb{R}_{>0})^n$ and $(\mathbb{R}_{<0})^n$ are chambers of $\mathfrak{D}_0(B)$, and we call them the \emph{all positive chamber} and the \emph{all negative chamber} respectively. If there exists a generic path which is finitely transverse from the all positive chamber to a chamber $\mathcal{C}$, then we call $\mathcal{C}$ a \emph{reachable chamber}. Moreover, if there exists a path which is finitely transverse from the all positive chamber crossing a wall $(\vv,W)$, then we call $(\vv,W)$ a \emph{reachable wall}.
\begin{lemma}[{\cite{GHKK}*{Lemma 2.10}}, \cite{nak21}*{Theorem 3.3}]\label{gvectortheorem}
Let $\Acal(B)$ be a cluster algebra of rank $n$ whose initial matrix is $B$. Each reachable chamber of $\mathfrak{D}_0(B)$ is expressed by the following form: 
\begin{align}
\mathbb{R}_{>0}\gg_1+\cdots+\mathbb{R}_{>0}\gg_n,
\end{align}
where $G=(\gg_1,\dots,\gg_n)$ is one of the $G$-matrices of $\Acal(B)$.
\end{lemma}

\subsection{Enough $g$-pairs property}
The definition of $g$-pairs was introduced in \cite{cl2} for cluster algebras with principal coefficients with the aim to study $d$-vectors. By our convention of $g$-vectors and Proposition \ref{independentP}, it generalizes to cluster algebras with arbitrary coefficients directly.

\begin{definition}\label{def:g-pair}
Let $\Acal(B)$ be a skew-symmetrizable cluster algebra, and $U$ be a subset of a non-labelled cluster $\xx$. Let $\xx_t$, $\xx_{t'}$ be two non-labeled clusters of $\Acal(B)$. 
The pair $(\xx_t , \xx_{t'})$ of clusters is called a \emph{$g$-pair} {associated with $U$ if it satisfies the
following two conditions.
\begin{itemize}
 \item [(a)] $U$ is a subset of $\xx_{t'}$ ;
 \item[(b)] Each row of $G_t^{B_{t'};t'}$ corresponding to $\xx\setminus U$ is a nonnegative vector.
\end{itemize}}
\end{definition}

{\begin{remark}
In \cite{cl2}, the $g$-pair is defined by using $R_t^{t'}:=(G_{t'}^{B;t_0})^{-1}G_{t}^{B;t_0}$ instead of $G_t^{B_{t'};t'}$. In fact, Definition \ref{def:g-pair} is equivalent to the original definition \cite{cl2}*{Definition 6}. This fact is given by \cite{cao}*{Corollary 7.18}.
\end{remark}}

Let $\Acal(B)$ be a cluster algebra of rank $n$ with the rooted vertex $t_0$, and $I=\{i_1,\dots,i_p\}$ be a subset of $\{1,2,\dots,n\}$. We say that a seed $\Sigma_t=(\xx_t,\yy_t, B_t)$ of $\Acal(B)$ is \emph{connected with $(\xx, \yy, B)$ by an $I$-sequence}, if there exists a composition of mutations $\mu_{k_s}\cdots\mu_{k_2}\mu_{k_1}$, such that $(\xx_t,\yy_t, B_t)=\mu_{k_s}\cdots\mu_{k_2}\mu_{k_1}(\xx, \yy, B)$, where $k_1,\dots,k_s\in I$.



\begin{theorem}[\cite{cl2}*{Theorems 8,9}]\label{enoughtheorem}
Let $\Acal(B)$ be a cluster algebra of rank $n$ with arbitrary coefficients. For any subset $U\subset\xx$ and any cluster $\xx_t$, there exists a seed $\Sigma_{t'}=(\xx_{t'},\yy_{t'},B_{t'})$ uniquely such that 
\begin{itemize}
  \item [(1)] the seed $\Sigma_{t'}=(\xx_{t'},\yy_{t'},B_{t'})$ is connected with the initial seed by an $I$-sequence, where $I=\{1,\dots,n\}\setminus J$ is the subset of $\{1,\dots,n\}$ such that $U=\{x_{j;t_0}\mid j\in J\}$.
  \item[(2)] $(\xx_t,\xx_{t'})$ is a $g$-pair associated with $U$.
\end{itemize}
\end{theorem}

We often refer to Theorem \ref{enoughtheorem} the {\it enough $g$-pairs property} of $\Acal(B)$.

{\begin{example}
We consider the cluster algebra in Example \ref{ex:prinA2}. We set $t=t_3$, that is, 
\begin{align*}
\xx_{t}=\left(\dfrac{x_1y_1y_2+y_1+x_2}{x_1x_2},\dfrac{y_1+x_2}{x_1}\right)
\end{align*}
The vertices in $\TT_2$ connected to $t_0$ are $t_{-1}$, $t_0$, and $t_1$.The $G$-matrices associated these vertices are given by Table \ref{g-mattt_-1t_1}.
\end{example}
\begin{table}[ht]
\begin{equation*}
\begin{array}{|c|c|c|c|}
\hline
&&&\\[-4mm]
t'& \hspace{0mm}t_{-1}&t_0 & t_1 \\
\hline
&&&\\[-3mm]
G_{t}^{B_t';{t'}}& \begin{bmatrix}
 1 & 1 \\
 -1 & 0
\end{bmatrix}& \begin{bmatrix}
 -1 & -1 \\
 0 & 1
\end{bmatrix}
& \begin{bmatrix}
 -1 & 0 \\
 0 & -1
\end{bmatrix}
 \\[4mm]
\hline
\end{array}
\end{equation*}
\caption{$G$-matrices $G_{t}^{B_{t'};t'}$}\label{g-mattt_-1t_1}
\end{table}
Thus, $(\xx_{t_3},\xx_{t_{0}})$ is the $g$-pair associated with $U=\{x_1\}$, and $(\xx_{t_3},\xx_{t_{-1}})$ is the $g$-pair associated with $U=\{x_2\}$.}
\section{$d$-vectors versus $f$-vectors}
Today, it is known that the $d$-vectors and the $f$-vectors are different vectors essentially, but it was pointed out that these two classes of vectors have similarities. Fomin and Zelevinsky expected that $d$-vectors of non-initial cluster variables coincide with $f$-vectors of them in any cluster algebras \cite{fziv}*{Conjecture 7.17}. A counterexample of it was given by \cite{fk}*{Example 6.7}, but it is known that this conjecture is true in cluster algebras of rank 2 and of finite type:
\begin{theorem}[\cite{g}*{Theorem 1.8, Remark 1.9}]\label{f=d}
\noindent
\begin{itemize}
\item[(1)]We fix an arbitrary cluster algebra of finite type. For any $i\in\{1,\dots,n\}$ and $t\in\TT_n$, we have the following relation:
\begin{align}\label{fdeq}
\ff_{i;t}=[\dd_{i;t}]_+,
\end{align}
\item[(2)]We fix an arbitrary cluster algebra of rank $2$. For any $i\in\{1,2\}$ and $t\in\TT_2$, we have the relation \eqref{fdeq}.
\end{itemize}
\end{theorem} 
In this section, we show a weaker similarity of these two families of vectors than \eqref{f=d} in general cluster algebras.

Fomin and Zelevinsky conjectured the following properties about $d$-vectors in \cite{fziv}*{Conjecture 7.4}, which was proved in \cite{cl2}:
\begin{theorem}[{\cite{cl2}*{Theorem 11}}]\label{dtheorem}
Let $\Acal(B)$ be any cluster algebra. The following statements hold:
\begin{itemize}
\item[(1)] A cluster variable $x_{i;t}$ is not one of the initial cluster variables if and only if $\dd_{i;t}$ is a non-negative vector.
\item[(2)] The $(i,j)$ entry $d_{ij;t}^{B_{t'};t'}$ of $D_{t}^{B_{t'};t'}$ equals the $(k,\ell)$ entry $d_{k\ell;s'}^{B_{s};s}$ of $D_{s'}^{B_{s};s}$ if $x_{i;t'}=x_{k;s'}$ and $x_{j;t}=x_{\ell;s'}$.
\item[(3)] There is a cluster containing $x_{i;t}(\neq x_k)$ and $x_k$ if and only if $d_{ki;t}=0$.
\end{itemize}
\end{theorem}
{In this section, we prove the following theorem, which is the $f$-vector version of Theorem \ref{dtheorem}.}
\begin{theorem}\label{ftheorem}
Let $\Acal(B)$ be any cluster algebra. The following statements hold:
\begin{itemize}
\item[(1)] A cluster variable $x_{i;t}$ is not one of the initial cluster variables if and only if $\ff_{i;t}$ is a non-zero vector.
\item[(2)] The $(i,j)$ entry $f_{ij;t'}^{B_{t};t}$ of $F_{t'}^{B_{t};t}$ equals the $(k,\ell)$ entry $f_{k\ell;s'}^{B_{s};s}$ of $F_{s'}^{B_{s};s}$ if $x_{i;t}=x_{k;s}$ and $x_{j;t'}=x_{\ell:s'}$.
\item[(3)] There is a cluster containing $x_{i;t}$ and $x_k$ if and only if $f_{ki;t}=0$.
\item[(4)]A cluster $\xx_t$ contains $x_k$ if and only if entries of the $k$th row of $F_t^{B:t_0}$ are all $0$.
\end{itemize}
\end{theorem}
\begin{remark}
Because of Proposition \ref{independentP}, it suffices to show Theorem \ref{ftheorem} in the case of cluster algebras with principal coefficients. In fact, for example, we assume that Theorem \ref{ftheorem} (1) holds in principal coefficient cases and there is a cluster algebra $\Acal_1(B)$ with coefficients in $\PP_1$ which does not satisfy Theorem \ref{ftheorem} (1). Then, there is a non initial cluster variable $x(1)$ whose $f$-vector is zero. Then, for any $i$, we have $x(1) \neq x_i(1)$ in $\Acal_1(B)$. Since $f$-vectors are independent of coefficients, we have a cluster variable $x$ whose $f$-vector is zero in $\Acal_{\bullet}(B)$ such that for any $i$, we have $x\neq x_i$ by Proposition \ref{independentP}. This is a contradiction. 
Therefore, we assume $\Acal(B)=\Acal_{\bullet}(B)$ in the proof of Theorem \ref{ftheorem}.
\end{remark}
The first statement implies there are no cluster variables whose expansions in the initial cluster variables are Laurent monomials except for the initial cluster variables, and this is a generalization of \cite{g}*{Corollaries 3.2, 4.2}. The second statement is important for defining the compatibility degree in the next section.
We prove Theorem \ref{ftheorem} in the rest of this section. 

To prove Theorem \ref{ftheorem} (1), we use the following two lemmas.

\begin{lemma}[{\cite{cl2}*{Lemma 2}}]\label{compatibility}
Let $\Acal(B)$ be any cluster algebra, and we fix any cluster variable $x$ of $\Acal(B)$. If, for all $i\in\{p+1,\dots,n\}$, there exists a cluster containing $x$ and $x_i$, then there exists a cluster containing $x$ and all the initial cluster variables $x_{p+1},\dots,x_{n}$.
\end{lemma}

\begin{lemma}[{\cite{cl2}*{Theorem 10}}]\label{pathconnect}
Let $\Acal(B)$ be any cluster algebra. We fix any subset $X$ of a cluster. Then, all seeds which have a cluster containing $X$ form a connected component of the exchange graph of $\Acal(B)$.
\end{lemma}



\begin{proof}[Proof of Theorem \ref{ftheorem} (1)]
The ``if" part is clear. We prove the ``only if "part by proving the contraposition of the statement. Since $\ff_{i;t}=0$ and by Proposition \ref{fconstant1}, we have $F_{i;t}(\yy)=1$. Then we have $x_{i;t}=\prod_j x_{j}^{g_{ji}}$ by the separation formula \eqref{eq:xjt=F/F}. We note that $\dd_{i;t}$ coincides with $-\gg_{i;t}$. According to Theorem \ref{dtheorem} (1), we have $\gg_{i;t}\in \ZZ_{\leq0}^n$ or $\gg_{i;t}=\ee_l$. We assume $\gg_{i;t}\in \ZZ_{\leq0}^n$. Without loss of generality, we can assume that
\begin{align*}
\gg_{i;t}=(a_1,\cdots, a_k,0,\cdots, 0)^\top\in \ZZ_{\leq 0}^n,\ a_1,\dots,a_k<0.
\end{align*}
By Theorem \ref{dtheorem} (3) and Lemma \ref{compatibility}, there exists a cluster $\xx$ such that $\xx$ contains $x_{i;t}$ and all the initial cluster variables $x_{k+1},\dots,x_{n}$. We set $\xx=(x_{i;t}, z_2,\dots, z_k, x_{k+1},\dots, x_{n})$. By the sign-coherence of the $G$-matrices (Theorem \ref{thm:signs-g}), the first $k$ components of $g$-vectors of $z_2,\dots, z_k$ lie in $\ZZ_{\leq 0}$. Then $G$-matrix associated with $\xx$ is a partitioned matrix $\begin{bmatrix} G&0\\ \ast&E_{n-k} \end{bmatrix}$, where each column of $G$ lies in $\ZZ_{\leq 0}^k$. We consider the cluster algebra $\Acal'$ by freezing all the initial cluster variables $x_{k+1},\dots, x_{n}$ of $\Acal(B)$. One can show that the $G$-matrix of the cluster $\tilde{\xx}=(x_{i;t}, z_2,\dots, z_k)$ of $\Acal'$ is $G$. We note that columns of $G$ are linearly independent because of $\det G=\pm1$. Then, according to Lemma \ref{gvectortheorem}, each column of $G$ is $-\ee_i$ and we have $k=1$. Thus we have $\gg_{i;t}=-\ee_1$ and $x_{i;t}=1/x_1$. We note that $\dd_{i;t}=\ee_{1}$ and thus there exists a cluster $\xx'$ such that $\xx'=\{1/x_1,x_2.\dots,x_n\}$ by Theorem \ref{dtheorem} (1) again. However, by Lemma \ref{pathconnect}, clusters containing $\{x_2,\dots,x_n\}$ are the initial cluster or a cluster which is obtained by mutating the initial cluster in direction $1$. Clearly, $\xx'$ is not the initial cluster. Since the numerator of $\xx'$ does not have any initial cluster variables, the entries of first column and the first row of $B$ are all 0. In this case, $\mu_1(x_1)=(y_1+1)/x_1$. Thus $\xx'\neq \mu_1(\xx_{t_0})$. This is a contradiction. Therefore, we have $\gg_{i;t}=-\dd_{i;t}=\ee_l$. Because of Theorem \ref{dtheorem} (1), we have $x_{i;t}=x_l$. This finishes the proof.
\end{proof}
Next, we prove Theorem \ref{ftheorem} (2). This statement follows from the fact that the $(i,j)$ entry of an $F$-matrix is invariant by mutations in direction $k$ such that $k\neq j$ and by initial-seed mutations in direction $\ell$ such that $\ell\neq i$. We prepare a lemma. This gives a recursion of the $F$-matrices by an initial-seed mutation.

Before describing the lemma, we introduce some notations. 
Let $J_{\ell}$ denote the $n\times n$ diagonal matrix obtained from the identity matrix $E_n$ by replacing the $(\ell, \ell)$ entry with $-1$.
For an $n\times n$ matrix $B=(b_{ij})$, let $[B]_+$ be the matrix obtained from $B$ by replacing every entry $b_{ij}$ with $[b_{ij}]_+$. 
Also, let $B^{k\bullet}$ be the matrix obtained from $B$ by replacing all entries outside of the $k$th row with zeros. Similarly, let $B^{\bullet k}$ be the matrix replacing all entries outside of the $k$th column.
Note that the maps $B\mapsto [B]_+$ and $B\mapsto B^{k\bullet}$ commute with each other, and the same is true for $B\mapsto [B]_+$ and $B\mapsto B^{\bullet k}$, so that the notations $[B]^{k\bullet}_+$ and $[B]^{\bullet k}_+$ make sense.
\begin{lemma}[{\cite{fg}*{Theorem 3.9}}]\label{initialseedmutation}

We set $\mu_k(B)=B_1$ and $\varepsilon \in \{\pm1\}$. We have 
\begin{gather} \label{eq:frs}
F^{B_1; t_1}_t=\big(J_k+[\varepsilon B]^{k\bullet}_+\big)F^{B; t_0}_t +\big[{-}\varepsilon G^{-B; t_0}_t\big]^{k\bullet}_+ +\big[\varepsilon G^{B; t_0}_t\big]^{k\bullet}_+.
\end{gather}
\end{lemma}

\begin{proof}[Proof of Theorem \ref{ftheorem} (2)]
Since $x_{i;t}=x_{k;s}$, according to Lemma \ref{pathconnect}, there exists a permutation $\sigma$ of indices such that $\sigma(k)=i$ and a vertex $s_0\in \mathbb{T}_n$ such that the seed $\Sigma_{s_0}$ is the permutation of $\Sigma_t$ by the permutation $\sigma$, that is
\[x_{u;s}=x_{\sigma(u);s_0}, y_{u;s}=y_{\sigma(u);s_0}, b_{uv;s}=b_{\sigma(u),\sigma(v);s_0}
\]
for all $u$ and $v$. Moreover, the seed $\Sigma_{s_0}$ is connected with $\Sigma_t$ by a $\{1,\dots,n\}\backslash\{i\}$-sequence. 

By definition of $f$-vectors, for any cluster variable $z$, the $f$-vector of $z$ with respect to the seed $\Sigma_{s_0}$ is the permutation of the $f$-vector of $z$ with respect to the seed $\Sigma_s$ by $\sigma$. In particular, $f_{k\ell;s'}^{B_s;s}=f_{i\ell;s'}^{B_{s_0};s_0}$. On the other hand, since $x_{j;t'}=x_{\ell;s'}$, we have $f_{i\ell;s'}^{B_t;t}=f_{ij;t'}^{B_t;t}$.
By Lemma \ref{initialseedmutation}, the initial-seed mutation at $m$ only change the $m$th row of $F$-matrices. Therefore, we have $f_{i\ell;s'}^{B_{s_0};s_0}=f_{i\ell;s'}^{{B}_{t};t}$. Putting all of these together, we obtain
\[f_{ij;t'}^{B_{t};t}=f_{k\ell;s'}^{B_{s};s}.
\]

\end{proof}
Let us prove Theorem \ref{ftheorem} (3). The following fact is known.
\begin{lemma}[\cite{cl2}*{Lemma 5.2}]\label{d-gtheorem}
 Suppose that $\mathcal A(B)$ is an arbitrary cluster algebra of rank $n$, and $({\bf x}_t,{\bf x}_{t^\prime})$ is a $g$-pair
 associated with $\{x_k\}$. 
 Let $G_t^{B_{t'};t'}=(g_{ij}')$ and $\dd^{B_{t'};t^\prime}_{i;t}=(d_{1}^\prime,\cdots,d_{n}^\prime)^\top$ be the $d$-vector of $x_{i;t}$ with respect to ${\bf x}_{t^\prime}$. We have that
\begin{itemize}
\item[(1)] $g_{ki}'>0$ if and only if $d_{k}^\prime=-1$, and if and only if $x_{i;t}\in{\bf x}_{t^\prime}$ and $x_{i;t}=x_{k;t^\prime}$.

\item[(2)] $g_{ki}'=0$ if and only if $d_{k}^\prime=0$, and if and only if $x_{i;t}\in{\bf x}_{t^\prime}$ and $x_{i;t}\neq x_{k;t^\prime}$.

\item[(3)] $g_{ki}'<0$ if and only if $d_{k}^\prime>0$, and if and only if $x_{i;t}\notin{\bf x}_{t^\prime}$.

\end{itemize}
\end{lemma}

We consider a similar lemma to Lemma \ref{d-gtheorem} for the $f$-vectors:

\begin{lemma}\label{f-gtheorem}
 Suppose that $\Acal(B)$ is an arbitrary cluster algebra of rank $n$, and $({\bf x}_t,{\bf x}_{t^\prime})$ is a $g$-pair 
 associated with $\{x_k\}$.
 Let $G_t^{B_{t'};t'}=(g_{ij}')$ and $\ff^{B_{t'};t^\prime}_{i;t}=(f_{1}^\prime,\cdots,f_{n}^\prime)^\top$ be the $f$-vector of $x_{i;t}$ with respect to ${\bf x}_{t^\prime}$. We have that
\begin{itemize}
\item[(1)] $g'_{ki}\geq 0$ if and only if $f_{k}^\prime=0$, and if and only if $x_{i;t}\in{\bf x}_{t^\prime}$.

\item[(2)] $g'_{ki}<0$ if and only if $f_{k}^\prime>0$, and if and only if $x_{i;t}\notin{\bf x}_{t^\prime}$.

\end{itemize}
\end{lemma}
Let us prove Lemma \ref{f-gtheorem}. While Lemma \ref{d-gtheorem} describes the relation between the $d$-vectors and the $g$-vectors, Lemma \ref{f-gtheorem} describes the relation between the $f$-vectors and the $g$-vectors. The $H$-matrix gives some relation of the $f$-vectors and the $g$-vectors.

\begin{definition} Let $B$ be an initial exchange matrix at $t_0$. Then, for any $t$, the $(i, j)$ entry of $H^{B; t_0}_t=\big(h_{ij; t}^{B;t_0}\big)$ is given by
\begin{gather*}
 u^{h_{ij; t}^{B;t_0}}=F^{B; t_0}_{j;t}|_{\text{Trop}(u)}\big(u^{[-b_{i1}]_+}, \dots, u^{-1}, \dots, u^{[-b_{in}]_+}\big) \qquad \big(\text{$u^{-1}$ in the $i$th position}\big).
\end{gather*}
The matrix $H^{B; t_0}_t$ is called the \emph{$H$-matrix} at $t$. \end{definition}

\begin{lemma} [\cite{fg}*{Theorem 3.7}] \label{prop:H=G}
For any $t\in\TT_n$, we have the following relation:
\begin{align} \label{eq:H=G}
H^{B; t_0}_t=-[-G^{B; t_0}_t]_+.
\end{align}
\end{lemma}

\begin{proof}[Proof of Lemma \ref{f-gtheorem}]
By Lemma \ref{d-gtheorem}, it suffices to show that $g'_{ki}<0$ implies $f_{k}^\prime>0$. By Lemma \ref{prop:H=G}, $g'_{ki}<0$ implies $h_{ki;t}^{B';t'}<0$. Then, we have $f_{k}^\prime>0$ by definition of the $H$-matrix.
\end{proof}

\begin{proof}[Proof of Theorem \ref{ftheorem} (3)]
We set $I=\{1,\dots,n\}\backslash \{k\}$. First, we prove the ``only if" part. Let $\xx_{t'}$ be a cluster such that $\xx_{t'}$ contains both $x_{i;t}$ and $x_k$. Then, according to Lemma \ref{pathconnect}, $\xx_{t'}$ is connected with $\xx_{t_0}$ by an $I$-sequence. We set $\xx_{t'}=\mu(\xx_{t_0})$. If we regard $\xx_{t'}$ as the initial cluster, then $f_{ki;t}^{B_{t'};t'}=0$. We can change the initial cluster from $\xx_{t'}$ to $\xx_{t_0}$ by initial-seed mutation induced by $\mu^{-1}$. Then, by Theorem \ref{ftheorem} (2), we have $f_{ki;t}=f_{ki;t}^{B_{t'};t'}=0$. Second, we prove the ``if" part. Let $\xx_s$ be a cluster containing the cluster variable $x_{i;t}$. By Theorem \ref{enoughtheorem}, there is a cluster $\xx_{s'}$ such that $(\xx_s,\xx_{s'})$ is a $g$-pair 
associated with $\{x_k\}$.
Then $x_{k;s'}=x_k$ holds because $\xx_{s'}$ is connected with $\xx_{t_0}$ by an $I$-sequence. Since $f_{ki;t}=0$ implies $f_{ki;t}^{B_{s'};s'}=0$ by Theorem \ref{ftheorem} (2), we have $x_{i;t}\in \xx_{s'}$ by Lemma \ref{f-gtheorem}. This finishes the proof.
\end{proof}
To prove Theorem \ref{ftheorem} (4), we use the following theorem:
\begin{theorem}[\cite{fg}*{Theorem 3.10}]\label{thm:F-opposite} For any exchange matrix~$B$ and $t_0, t \in \TT_n$, we have
\begin{gather}\label{eq:F-opposite}
\big(F_t^{B;t_0}\big)^\top = F_{t_0}^{B_t^\top;t}.
\end{gather}
\end{theorem}
We call this property the \emph{self-duality} for the $F$-matrices. Theorem \ref{thm:F-opposite} implies the rows of the $F$-matrices are the $f$-vectors of another cluster algebra. 
\begin{proof}[Proof of Theorem \ref{ftheorem} (4)]
The ``only if" part follows from the ``only if" part of Theorem \ref{ftheorem} (3). We prove the ``if" part. By Theorem \ref{thm:F-opposite}, the transposition of the $F$-matrix $F_t^{B;t_0}$ is another $F$-matrix $F_{t_0}^{B_t^\top;t}$. By assumption, the $k$th column of $F_{t_0}^{B_t^\top;t}$ is the zero vector. By Theorem \ref{ftheorem} (1), a cluster variable of $\mathcal A(B_t^\top)$ associated with this column is an initial cluster variable. Then, by Theorem \ref{ftheorem} (3), there is a $j\in\{1,\dots,n\}$ such that all entries of the $j$th row of $F_{t_0}
^{B_t^\top;t}$ are 0. This implies the $j$th column of $F_t^{B;t_0}$ is the zero vector. Therefore, all entries of the $j$th column and the $k$th row of $F_t^{B;t_0}$ are all 0. By Theorem \ref{ftheorem} (1), $\xx_t$ has at least one initial cluster variable. We show that one of these initial cluster variable is $x_k$. We assume that $x_k\not\in\xx_t$. Then, $\xx_t$ has an initial cluster variable which is not $x_k$. We assume that this cluster variable is $x_{k'}$. Then by Theorem \ref{ftheorem} (3), the $k'$th column of $F_{t_0}^{B_t^\top;t}$ is the zero vector. In the same way as the previous argument, there exists a $j'\in \{1,\dots,n\}$ such that the $j'$th column of $F_{t}^{B;t_0}$ is the zero vector. Since cluster variables in a cluster are algebraically independent, we note that $j\neq j'$. Therefore, all entries of the $j,j'$th columns and the $k,k'$th rows of $F_{t}^{B;t_0}$ are all 0. By Theorem \ref{ftheorem} (1) again, $\xx_t$ has at least two initial cluster variables. By assumption, $\xx_t$ has a cluster variable which is neither $x_k$ nor $x_{k'}$. By repeating this argument, we have $F_{t}^{B;t_0}=0$. Therefore by Theorem \ref{ftheorem} (1), $\xx_t=\xx_{t_0}$. This conflicts with the assumption. 
\end{proof}


{Theorem \ref{ftheorem} plays a fundamental role in Section 4 to define the compatibility degree and to prove its properties.}

By these theorems, we have the following corollary:

\begin{corollary}\label{f-dcorrespondence}
For any cluster algebra $\Acal(B)$, $f_{ij;t}=0$ if and only if $d_{ij;t}=0$ or $-1$.
\end{corollary}
\begin{proof}
It follows from Theorem \ref{dtheorem} (1),(3) and Theorem \ref{ftheorem} (3).
\end{proof}

The property of $d$-vectors corresponding to Theorem \ref{ftheorem} (4) has not been known. However, we obtain it by using Theorem \ref{ftheorem} (4) and Corollary \ref{f-dcorrespondence}. 

\begin{corollary}
For any cluster algebra $\Acal(B)$, all entries of the $k$th row of $D_{t}^{B;t_0}$ are all non-positive if and only if $\xx_t$ contains $x_k$.
\end{corollary}
\begin{proof}
The ``if" part follows from Theorem \ref{dtheorem} (3). We show the ``only if" part. By Corollary \ref{f-dcorrespondence}, the $k$th row of $F_t^{B;t_0}$ are all 0. By Theorem \ref{ftheorem} (4), $\xx_t$ contains $x_k$. This finishes the proof.
\end{proof}

\section{Compatibility degree and its properties}
The classical compatibility degree was introduced to define the generalized associahedron. This is a function on the set of pairs of roots, and the generalized associahedra are simplicial complexes whose simplexes are sets consisting of roots such that each classical compatibility degree of a pair of roots is 0. In this section, we generalize it to a function on pairs of cluster variables in a different way from \cite{cl2} by using $f$-vectors and give some properties of the generalized one, the compatibility degree. 

\subsection{Classical compatibility degree and generalized associahedra}
In this subsection, we explain the \emph{classical compatibility degree} and the \emph{generalized associahedra} introduced in \cite{fzy}. 
Let $\Phi$ be a root system of finite type. We denote by $\Phi_{\geq -1}$ the set of almost positive roots, that is, the union of all negative simple roots and all positive roots. Let $C_\Phi$ be a Cartan matrix corresponding to $\Phi$ and $\Gamma_\Phi$ be a Dynkin graph corresponding to $\Phi$. {Denote by $I$ the set of vertices of $\Gamma_{\Phi}$ and $(I_-,I_+)$ the bipartition of $\Gamma_{\Phi}$.} 
Next, we define $t_\pm$ which are compositions of simple reflections as follows:
\begin{align}
t_+ = \prod_{i\in I_+} s_i,\quad t_- = \prod_{i\in I_-} s_i.
\end{align}
We define transformations $\tau_\pm\colon \Phi_{\geq-1}\to \Phi_{\geq-1}$ as follows:
\begin{align}
\tau_+(\alpha)&=
\begin{cases}\alpha \quad &\text{if }\alpha =-\alpha _{j},\ j\in I_-; \\ 
t_+\left( \alpha \right)&\text{otherwise},
\end{cases}\\
\tau_-(\alpha)&=
\begin{cases}\alpha \quad &\text{if }\alpha =-\alpha _{j},\ j\in I_+; \\ 
t_-\left( \alpha \right)&\text{otherwise}.
\end{cases}
\end{align}

For $k\in\ZZ$ and $i\in I$, we abbreviate
\begin{align}
\alpha(k;i) = (\tau_- \tau_+)^k( - \alpha_i).
\end{align}
In particular, $\alpha(0;i)=-\alpha_i$ for all~$i$
and $\alpha(\pm 1;i)=\alpha_i$ for $i\in I_\mp\,$.

Let $h$ be the Coxeter number of $\Phi$ and $w_\circ$ be the longest element of the Weyl group of $\Phi$. Let $i \mapsto i^*$ denote
the involution on $I$ defined by $- \alpha_{i^*}:=w_\circ (\alpha_i) $.
It is known that this involution preserves each of the sets $I_+$ and $I_-$
when $h$ is even, and interchanges them when $h$ is odd.

\begin{proposition}[\cite{fzy}*{Proposition 2.5}]
\label{pr:tau orbits}
\noindent
\begin{itemize}
\item [(1)]
Suppose $h = 2e$ is even.
Then the map $(k,i) \mapsto \alpha(k;i)$
restricts to a bijection
\begin{align}
[0,e] \times I \to \Phi_{\geq -1}.
\end{align}
Furthermore, $\alpha(e+1;i) = - \alpha_{i^*}$
for any $i$.

\item[(2)]
Suppose $h = 2e+1$ is odd.
Then the map $(k,i) \mapsto \alpha(k;i)$
restricts to a bijection
\begin{align}
([0,e+1] \times I_-) \cup ([0,e] \times I_+)\to
\Phi_{\geq -1}.
\end{align}
Furthermore, $\alpha(e+2;i) = - \alpha_{i^*}$
for $i\in I_-$,
and $\alpha(e+1;i) = - \alpha_{i^*}$
for~$i\in I_+$.
\end{itemize}
\end{proposition}
By this proposition, we can express any root $\beta\in \Phi_{\geq -1}$ with $\beta=\tau(-\alpha_i)$, where $\tau$ is a composition of $\tau_+$ and $\tau_-$, and $-\alpha_i$ is a negative simple root.
We consider a function $(\cdot \parallel \cdot)_{\rm{cl}}\colon\Phi_{\geq -1}\times\Phi_{\geq -1}\to \ZZ_{\geq 0}$ characterized by the following property:
For any negative simple root $-\alpha_i$ and any root $\beta$, we have
\begin{align}
(-\alpha_i \parallel \beta)_{\rm{cl}}=[(\beta:\alpha_i)]_+,
\end{align}
and for any roots $\alpha,\beta$, we have
\begin{align}
(\alpha\parallel \beta)_{\rm{cl}}=(\tau_\varepsilon(\alpha)\parallel\tau_\varepsilon(\beta))_{\rm{cl}},
\end{align}
where$(\beta:\alpha_i)$ is the coefficient integer of $\alpha_i$ in the expansion of $\beta$ in simple roots, and $\varepsilon\in\{\pm1\}$. This function is well-defined by Proposition \ref{pr:tau orbits}. It is called the \emph{classical compatibility degree}. In \cite{fzy}, the classical compatibility degree is called simply the \emph{compatibility degree}, but we adopt this name in imitation of \cite{cp} to distinguish it from the other forthcoming degrees. For $\alpha,\beta \in \Phi_{\geq-1}$, we say that $\alpha$ and $\beta$ are \emph{compatible} if $(\alpha\parallel \beta)_{\rm{cl}}=(\beta\parallel \alpha)_{\rm{cl}}=0$. 

By using the classical compatibility degree, we define the generalized associahedra.

\begin{definition}
For a root system $\Phi$, we define the \emph{generalized associahedron} $\Delta(\Phi)$ as a simplicial complex whose simplexes are subsets of almost positive roots such that their elements are pairwise compatible. 
\end{definition}
\begin{example}
We consider the root system of type $A_2$. We give a generalized associahedron of type $A_2$ in Figure \ref{A2associahedra}. We remark that this complex is isomorphic to the cluster complex given in Example \ref{clustercomplexofA2}. We introduce the correspondence between cluster complexes and generalized associahedra in Theorems \ref{rootbijection} and \ref{cluster-rootidentification}.
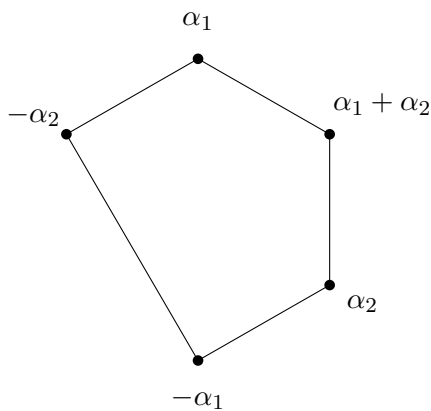
\begin{figure}[ht]
\caption{Generalized associahedron of type $A_2$}\label{A2associahedra}
\[
\begin{tikzpicture}
 \coordinate (0) at (0,0) ;
 \coordinate (1) at (30:2) ;
 \coordinate (2) at (90:2);
 \coordinate (3) at (150:2);
 \coordinate (4) at (270:2);
 \coordinate (5) at (330:2);
 \node (6) at (30:2.8) {$\alpha_1+\alpha_2$};
 \node (7) at (90:2.5){$\alpha_1$};
 \node (8) at (150:2.5){$-\alpha_2$};
 \node (9) at (270:2.5){$-\alpha_1$};
 \node (10) at (330:2.5){$\alpha_2$};
 \fill(1) circle (0.7mm);
 \fill(2) circle (0.7mm);
 \fill(3) circle (0.7mm);
 \fill(4) circle (0.7mm);
 \fill(5) circle (0.7mm);
 \draw(1) to (2);
 \draw(2) to (3);
 \draw(3) to (4);
 \draw(4) to (5);
 \draw(5) to (1);
\end{tikzpicture}
\]
\end{figure}
\end{example}
The function $(\cdot\parallel\cdot)_{\rm{cl}}$ can be regarded as a function on cluster variables of a cluster algebra in the following way: We fix a root system $\Phi$ and the sign of vertices $I=I_+\cup I_-$ of $\Gamma_{\Phi}$. We denote by $B(C_\Phi)=(b_{ij})$ a skew-symmetrizable matrix obtained from the Cartan matrix $C_{\Phi}=(C_{ij})$ by the following equation:
\begin{align*}
b_{ij}=\begin{cases}
0\quad &\text{if } i=j,\\
-\varepsilon C_{ij} & \text{if } i\neq j \text{and } i\in I_\varepsilon.
\end{cases}
\end{align*}
We call $\Acal(B(C_\Phi))$ the cluster algebra induced by $\Phi$.
\begin{theorem}[\cite{fzii}*{Theorem 1.9}]\label{rootbijection}
For a root system $\Phi$, there is a unique bijection $\alpha\mapsto x[\alpha]$ from $\Phi_{\geq-1}$ to the set $\Xcal$ of all cluster variables in $\Acal(B(C_\Phi))$, such that, for any $\alpha=\sum_{i}a_i\alpha_i \in\Phi_{\geq-1}$, the cluster variable $x[\alpha]$ is expressed in terms of the initial cluster $x_1,\dots,x_n$ as
\begin{align}
x[\alpha]=\dfrac{P(x_1,\dots,x_n)}{x_1^{a_1}\cdots x_n^{a_n}},
\end{align}
where $P(x_1,\dots,x_n)$is a polynomial over $\ZZ\PP$ which is not divisible by any $x_i$. Under this bijection, $x[-\alpha_i] =x_i$.
\end{theorem}
The bijection in Theorem \ref{rootbijection} is natural in the sense of the following.

\begin{theorem}[\cite{fzii}*{Theorem 1.12}]\label{cluster-rootidentification}
Under the bijection of Theorem \ref{rootbijection}, the cluster complex $\Delta(\Acal(B(C_{\Phi})))$ is identified with the simplicial complex $\Delta(\Phi)$. 
\end{theorem}

By Theorems \ref{rootbijection} and \ref{cluster-rootidentification}, we can identify almost positive roots in $\Phi_{\geq -1}$ with the $d$-vectors of cluster variables of $\Acal(B(C_\Phi))$. By abusing of notation, we use $(\cdot \parallel \cdot)_{\rm{cl}}$ as a function on $\Xcal \times \Xcal$.
We denote by $\Phi^\vee$ the dual root system of $\Phi$, and for any $\alpha\in\Phi$, we denote by $\alpha^{\vee}\in\Phi^\vee$ the coroot of $\alpha$. By definition, it is clear that $\Acal(B(C_{\Phi^\vee}))$ is $\Acal(B(C_\Phi)^\top)$ or $\Acal(-B(C_\Phi)^\top)$ (depending on the choice of $I_+$). We remark that the classical compatibility on $\Xcal \times \Xcal$ depends only on root systems, therefore we can assume $\Acal(B(C_{\Phi^\vee}))=\Acal(-B(C_\Phi)^\top)$ without loss of generality.

The classical compatibility degree satisfies the following property:

\begin{proposition}[\cite{fzy}*{Proposition 3.3}]
\label{pr:compatibility-symmetry}
We fix $\Phi$ and an induced cluster algebra $\Acal(B(C_\Phi))$. 
\begin{itemize}
\item[(1)]
We have $(x[\alpha] \parallel x[\beta])_{\rm{cl}} = (x[\beta^\vee] \parallel x[\alpha^\vee])_{\rm{cl}}$
for every $\alpha, \beta \in \Phi_{\geq -1}$.\\
In particular, if $\Phi$ is simply-laced, then
$(x[\alpha] \parallel x[\beta]) _{\rm{cl}}= (x[\beta] \parallel x[\alpha])_{\rm{cl}}$.

\item[(2)]
If $(x[\alpha] \parallel x[\beta])_{\rm{cl}} = 0$, then
$(x[\beta] \parallel x[\alpha])_{\rm{cl}} = 0$.

\item[(3)]
If $\alpha$ and $\beta$ belong
to $\Phi(J)_{\geq -1}$ for some proper subset $J \subset I$,
then their compatibility degree with
respect to the root subsystem $\Phi(J)$ is equal to $(x[\alpha] \parallel x[\beta])_{\rm{cl}}$.
\end{itemize}
\end{proposition}
We call (1) the \emph{duality property}, (2) the \emph{symmetry property}, and (3) the \emph{embedding property} respectively. 
Moreover, the classical compatibility degree satisfies the following two properties, the \emph{compatibility property} and the \emph{exchangeability property}:
\begin{proposition}\label{classical-compatibility}
Let $\Phi$ be a root system and $\Acal(B(C_\Phi))$ be an induced cluster algebra by $\Phi$. For any set of cluster variables $X$, there exists a cluster $\xx$ such that $\xx$ contains $X$ if and only if the classical compatibility degrees of any pairs of cluster variables in $X$ are $0$.
\end{proposition}
\begin{proof}
It follows from Theorem \ref{cluster-rootidentification} immediately.
\end{proof}

\begin{proposition}\label{classical-exchangeability}
Let $\Phi$ be a root system and $\Acal(B(C_\Phi))$ be an induced cluster algebra by $\Phi$. For any $x[\alpha],x[\beta]$, there exists a set $X$ of cluster variables such that $X\cup x[\alpha]$ and $X\cup x[\beta]$ are both clusters, if and only if $(x[\alpha]\parallel x[\beta])_{\rm{cl}}=(x[\beta]\parallel x[\alpha])_{\rm{cl}}=1$.
\end{proposition}
\begin{proof}
The exchangeability of almost positive roots is proved by \cite{cfz}*{Lemma 2.2} and \cite{fzii}*{Corollary 4.4}. The proposition is shown by combining it with Theorem \ref{cluster-rootidentification}. 
\end{proof}
We consider a natural generalization of the classical compatibility degree preserving these properties in the next subsection.
\subsection{Compatibility degree}
We introduce the compatibility degree. This is defined by using components of $f$-vectors. In this subsection, we prove that compatibility degree preserves the results of Proposition \ref{pr:compatibility-symmetry} and Proposition \ref{classical-compatibility}.
\begin{definition}\label{f-compatibility}
Let $\Acal(B)$ be a cluster algebra. We define the \emph{compatibility degree} $(\cdot\parallel \cdot) \colon \Xcal \times \Xcal \to \ZZ_{\geq 0}$ of $\Acal(B)$ as follows: For any two cluster variables $x$ and $x'$, if $x=x_{i;t}$ and $ x'=x_{j;t'}$, then 
\begin{align}
(x\parallel x')=f_{ij;t'}^{B_t;t}.
\end{align}
When we want to emphasize that this function is defined by $f$-vector, we use $(x\parallel x')_f$ as the notation.
\end{definition}
We remark that the choice of $i,j,t,t'$ satisfying $x=x_{i;t}$ and $x'=x_{j;t'}$ is not unique, but the compatibility degree is well-defined by Theorem \ref{ftheorem} (2).
This function is a generalization of the classical compatibility degree.
\begin{theorem}\label{classical-fcorrespondence}
We fix any root system $\Phi$ and its induced cluster algebra $\Acal(B(C_\Phi))$. For any cluster variables $x$ and $x'$, we have 
\begin{align}
(x\parallel x')_{\rm{cl}}=(x\parallel x').
\end{align}
\end{theorem}
To prove the theorem, we introduce $d$-compatibility degree defined by \cite{cl2}. 
\begin{definition}
Let $\Acal(B)$ be a cluster algebra. We define the \emph{$d$-compatibility degree} $(\cdot\parallel \cdot)_d \colon \Xcal \times \Xcal \to \ZZ_{\geq 0}$ of $\Acal(B)$ as follows: For any two cluster variables $x$ and $x'$, if $x=x_{i;t}$ and $ x'=x_{j;t'}$, then 
\begin{align}
(x\parallel x')_d=\left[d_{ij;t'}^{B_t;t}\right]_+.
\end{align}
\end{definition}
In \cite{cl2}, the compatibility degree is not defined by $\left[d_{ij;t'}^{B_t;t}\right]_+$ but $d_{ij;t'}^{B_t;t}$. We adopt this definition for simplicity of the notation. 
The following theorem is essential for the proof of Theorem \ref{classical-fcorrespondence}:
\begin{theorem}[\cite{cp}*{Corollary 3.2}]\label{classical-dcorrespondence}
We fix any root system $\Phi$ and its induced cluster algebra $\Acal(B(C_\Phi))$. For any cluster variables $x$ and $x'$, we have 
\begin{align}
(x\parallel x')_{\rm{cl}}=(x\parallel x')_d.
\end{align}
\end{theorem}
\begin{proof}[Proof of Theorem \ref{classical-fcorrespondence}]
Since $\Acal(B(C_\Phi))$ is of finite type, it follows from Theorem \ref{f=d} and Theorem \ref{classical-dcorrespondence}.
\end{proof}
\begin{remark}
{Theorem \ref{classical-fcorrespondence} can be generalized from the classical compatibility degree to \emph{$c$-compatiblity degree}, which is a function of the set of pairs of almost positive root in finite root system, defined by \cite{rea}*{Proposition 7.2}.} This fact follows from \cite{cp}*{Corollary 3.3} and Theorem \ref{f=d}. 
\end{remark}
We will show that the compatibility degree satisfies properties which are analogous to Proposition \ref{pr:compatibility-symmetry} and Proposition \ref{classical-compatibility}.
First, we consider the following proposition, which is an analogue of Proposition \ref{pr:compatibility-symmetry}.
\begin{proposition}
\label{pr:f-compatibility-symmetry}
We fix any cluster algebra $\Acal(B)$ of rank $n$. For any $x=x_{i;t}$, we denote by $x^{\vee}$ the $i$th cluster variable in the cluster at $t$ of $\Acal(-B^\top)$. 
\begin{itemize}
\item[(1)]
For any two cluster variables $x, x'$, we have $(x \parallel x') = ((x')^\vee \parallel x^\vee)$. In particular, if $B$ is skew-symmetric, then
$(x \parallel x') = (x' \parallel x)$.

\item[(2)]
If $x=x_{i;t}$ and $x'=x_{j;t'}$, we have $(x \parallel x') = s_i^{-1}s_j(x' \parallel x)$, where $s_i$ is the $i$th entry of skew-symmetrizer $S$ of $B$. In particular, if $(x\parallel x')= 0$, then we have
$(x' \parallel x) = 0$.

\item[(3)]
 Let $J=\{k_1,\dots k_m\}$ be a subset of $\{1,\dots, n\}$ and $B_J$ be the submatrix of $B$ such that $B_J=(b_{k_{i}k_{j}})$. For any pair of cluster variables $x,x'$ of $\Acal(B_J)$, which we regard as a pair of cluster variables of $\Acal(B)$ by embedding, $(x\parallel x')$ on $\Acal(B_J)$ equals to $(x\parallel x')$ on $\Acal(B)$.

\end{itemize}
\end{proposition}
To prove Proposition \ref{pr:f-compatibility-symmetry}, we prepare two lemmas:
\begin{lemma}[\cite{fg}*{Theorem 2.8}] \label{thm:F=F}
For any exchange matrix~$B$ and $t_0, t \in \TT_n$, we have
\begin{align}
\label{eq:F=F}
F_t^{-B;t_0}&=F_t^{B;t_0}.
\end{align}
\end{lemma}
\begin{lemma}\label{lem:F-symmetry}
For any exchange matrix~$B$ and $t_0, t \in \TT_n$, we have
\begin{align}\label{eq:F-symmetry}
F_{t}^{B;t_0}=S^{-1}F_t^{-B^\top;t_0}S.
\end{align}
\end{lemma}
\begin{proof}
By \eqref{f-recursion}, we have a recursion of the $F$-matrices
\begin{align}\label{eq:fmat-frontmutation}
 F^{B;t_0}_{t'}=F^{B;t_0}_t J_{\ell}+\max([C^{B; t_0}_t]^{\bullet \ell}_+ +F^{B; t_0}_t[B_t]^{\bullet \ell}_+, [-C^{B; t_0}_t]^{\bullet \ell}_+ +F^{B; t_0}_t[-B_t]^{\bullet \ell}_+).
\end{align}
By \cite{nz}*{(2.7)} and definition of $S$, for any $t$, we have
\begin{align}
\label{eq:C-symmetry}
C_{t}^{B;t_0}&=S^{-1}C_t^{-B^\top;t_0}S ,\\
\label{eq:B-symmetry}
B_t&=S^{-1}(-B_t^{\top})S.
\end{align}
We prove \eqref{eq:F-symmetry} by induction on the distance from $t_0$ to $t$. If $t=t_0$, then \eqref{eq:F-symmetry} holds clearly because $F_{t_0}^{B;t_0}$ is the zero matrix. When we assume \eqref{eq:F-symmetry} holds at $t$, we have \eqref{eq:F-symmetry} holds at $t'$ by substituting \eqref{eq:fmat-frontmutation} with \eqref{eq:C-symmetry} and \eqref{eq:B-symmetry}.
\end{proof}
\begin{proof}[Proof of Proposition \ref{pr:f-compatibility-symmetry}]
First, we prove (1). By Theorem \ref{thm:F-opposite} and Lemma \ref{thm:F=F}, for any $t$ and $t'$, we have
\begin{align}
F_{t'}^{B_t;t}=\left( F_{t}^{-B_{t'}^\top;t'}\right)^\top.
\end{align}
Thus, we have 
\begin{align}
f_{ij;t'}^{B_t;t}=f_{ji;t}^{-B_{t'}^\top;t'}.
\end{align}
This implies the first statement of (1). Furthermore, if $B$ is skew-symmetric, then we have $B_{t'}=-B_{t'}^\top$. This implies the second statement of (1). Second, we prove (2). By (1) and Lemma \ref{lem:F-symmetry}, for any $t$ and $t'$, we have
\begin{align}
F_{t'}^{B_t;t}=\left( F_{t}^{-B_{t'}^\top;t'}\right)^\top=\left( SF_{t}^{B_{t'};t'}S^{-1}\right)^\top=S^{-1}\left(F_{t}^{B_{t'};t'}\right)^\top S.
\end{align}
Thus, we have 
\begin{align}
f_{ij;t'}^{B_t;t}=s_i^{-1}s_jf_{ji;t}^{B_{t'};t'}.
\end{align}
This implies (2). Finally, we prove (3). Without loss of generality, we can assume $J=\{1,\dots,m\}$. It suffice to show that $F_t^{B_J;t_0}$ equals the $m\times m$ principal submatrix of $F_t^{B;t_0}$ for any $t_0\in \TT_n$ and $t\in \TT_m$, where $\TT_m$ is the $m$-regular graph whose labels of edges are $1,\dots,m$ and which is a connected component of $\TT_n$ containing $t_0$. 
We prove (3) by induction on the distance from $t_0$ to $t$. The base case $t=t_0$ is immediate as $F_{t_0}^{B;t_0}=0$ and $F_{t_0}^{B_J;t_0}=0$. Let $C_t^{B_J;t_0}=(\overline{c}_{ij;t})$ and we abbreviate $F_{i;t}^{B;t_0}(\yy)=F_{i;t}$ and $F_{i;t}^{B_J;t_0}(\yy)=\overline{F}_{i;t}$. We have the following fact by direct calculation: ${B_J}_t$ equals the $m\times m$ principal matrix $B_t$, $F_{i;t}=1$ for all $i\in\{m+1,\dots,n\}$, and the left side $m\times n$ submatrix of $C_t^{B;t_0}$ is $\begin{bmatrix} C_t^{B_J;t_0}\\ 0\end{bmatrix}$. By these facts and the inductive assumption, for \begin{xy}(0,1)*+{t}="A",(10,1)*+{t'}="B",\ar@{-}^\ell "A";"B" \end{xy}, we have
\begin{align*}
\overline{F}_{i;t'}&= \overline{F}_{i;t}=F_{i;t} =F_{i;t'} \quad \text{if $i\neq
 \ell$,}\\
\overline{F}_{\ell;t'} &= \frac{\displaystyle\prod_{j=1}^{m} y_j^{[\overline{c}_{j\ell;t}]_+}
\prod_{i=1}^{m} \overline{F}_{i;t}^{[\overline{b}_{i\ell;t}]_+}
+
\prod_{j=1}^{m} y_j^{[-\overline{c}_{j\ell;t}]_+}
\prod_{i=1}^{m} \overline{F}_{i;t}^{[-\overline{b}_{i\ell;t}]_+}}{\overline{F}_{\ell;t}}\\
&= \frac{\displaystyle\prod_{j=1}^{n} y_j^{[c_{j\ell;t}]_+}
\prod_{i=1}^{n} F_{i;t}^{[b_{i\ell;t}]_+}
+
\prod_{j=1}^{n} y_j^{[c_{j\ell;t}]_+}
\prod_{i=1}^n F_{i;t}^{[-b_{i\ell;t}]_+}}{F_{\ell;t}}\\
&={F}_{\ell;t'}.
\end{align*}
Therefore, $F_t^{B_J;t_0}$ equals the $m\times m$ principal submatrix of $F_t^{B;t_0}$. 
\end{proof}
\begin{remark}
We can prove the second statement of (1) by using the first statement of (2) because when $B$ is skew-symmetric, then $s_i^{-1}s_j$ is always 1. 
\end{remark}
Next, we consider the \emph{compatibility property}, which is an analogue of Proposition \ref{classical-compatibility}.
\begin{theorem}\label{fcompatibilityproperty}
For any cluster algebra $\Acal(B)$ and any set $X$ of cluster variables, there exists a cluster $\xx$ such that $\xx$ contains $X$ if and only if the compatibility degrees of any pairs of cluster variables in $X$ are $0$.
\end{theorem}
\begin{proof}
It follows from Theorem \ref{ftheorem} (3) and Lemma \ref{compatibility} immediately.
\end{proof}

Let us compare the compatibility degree with the $d$-compatibility degree. It is proved by \cite{cl2} that $d$-compatibility degree also has the similar property of Theorem \ref{fcompatibilityproperty}.
\begin{theorem}[\cite{cl2}*{Theorem 7.4}]\label{dcompatibilityproperty}
For any cluster algebra $\Acal(B)$ and any set $X$ of cluster variables, there exists a cluster $\xx$ such that $\xx$ contains $X$ if and only if the $d$-compatibility degrees of any pairs of cluster variables in $X$ are $0$.
\end{theorem}

However, the $d$-compatibility degree does not satisfy the similar property of the duality and symmetry properties (Proposition \ref{pr:f-compatibility-symmetry} (1),(2)). 
Actually, if these properties hold for $d$-vectors, the $D$-matrices must satisfy the following equation when $B$ is skew-symmetric:
 \begin{gather}\label{eq:D-opposite}
\big(D_t^{B;t_0}\big)^\top = D_{t_0}^{-B_t^\top;t}= D_{t_0}^{B_t;t}.
\end{gather}

However, this equation does not hold generally unlike the $F$-matrices.
For the class of cluster algebras arising from marked surfaces, \cite{rs} gave a complete classification of marked surfaces whose corresponding cluster algebras satisfy \eqref{eq:D-opposite}.
\begin{theorem}[\cite{rs}*{Theorem 2.4}]
 \label{D surface}
 The equation \eqref{eq:D-opposite}
 holds for the cluster algebra arising from a marked surface if and only if the marked surface is one of the following.
 \begin{itemize}
 \item[(1)]
 \label{disk good}
 A disk with at most one puncture (finite types A and D).
 \item[(2)]
 \label{small Atilde good}
 An annulus with no punctures and one or two marked points on each boundary component (affine types $\tilde A_{1,1}$, $\tilde A_{2,1}$, and $\tilde A_{2,2}$).
 \item[(3)]
 \label{small Dtilde good}
 A disk with two punctures and one or two marked points on the boundary component (affine types $\tilde D_3$ and $\tilde D_4$).
 \item[(4)]
 \label{sphere 4 good}
 A sphere with four punctures and no boundary components.
 \item[(5)]
 \label{torus 1 good}
 A torus with exactly one marked point (either one puncture or one boundary component containing one marked point).
 \end{itemize}
\end{theorem}
According to Theorem \ref{D surface}, for example, cluster algebras arising from a disk with three punctures and one marked point on the boundary component do not satisfy \eqref{eq:D-opposite}. Let us see a concrete example.
\begin{example}\label{counterexsym}
We set $\PP=\{1\}$ the trivial semifield and consider a seed $(\xx,B)$, where 
\begin{align*}
\xx=(x_1,x_2,x_3,x_4,x_5,x_6,x_7), \quad
 B=
 \begin{bmatrix}
 0&0&-1&0&1&0&0\\
 0&0&-1&0&1&0&0\\
 1&1&0&-1&-1&1&0\\
 0&0&1&0&0&-1&1\\
 -1&-1&1&0&0&-1&1\\
 0&0&-1&1&1&0&-1\\
 0&0&0&-1&-1&1&0
 \end{bmatrix}.
\end{align*}
Moreover, we set 
\begin{align*}
\xx'=(x_1,x_2',x_3',x_4',x_5',x_6,x_7')=\mu_7\mu_5\mu_4\mu_3\mu_2(\xx).
\end{align*}
Then, we have $(x_2 \parallel x_7')_d=2,(x_7^{'\vee} \parallel x_2^{\vee})_d=(x'_7 \parallel x_2)_d=1$. Let us see this fact by using marked surface and their flips (cf. \cite{fst} or Section \ref{subsection:markedsurface} in this paper). $\Acal(B$) is a cluster algebra arising from the marked surface in Figure \ref{Bsurface}.
\begin{figure}[ht]
 \centering
 \caption{Marked surface corresponding to $B$ \label{Bsurface}}
 \vspace{-1cm}\hspace{5mm}
\begin{tikzpicture}[baseline=0mm,scale=0.7]
 \coordinate (d) at (0,-2);
 \coordinate (cd) at (0,-0.5);
 \coordinate (cl) at (150:1);
 \coordinate (cr) at (30:1);
 \draw (0,0) circle (2);
 \draw (d) to node[fill=white,inner sep=1]{$1$} (cd);
 \draw (d) to [out=-50,in=-100,relative] node[pos=0.85]{\rotatebox{60}{\footnotesize $\bowtie$}} node[fill=white,inner sep=1,pos=0.55]{$2$} (cd);
 \draw (d) .. controls (-2,0.5) and (0.5,0) .. 
 node[fill=white,inner sep=1,pos=0.55]{$3$} (cr);
 \draw (d) to [out=50,in=100,relative] 
 node[fill=white,inner sep=1,pos=0.55]{$4$} (cl);
 \draw (d) to [out=-50,in=-100,relative] 
 node[fill=white,inner sep=1,pos=0.55]{$5$} (cr);
 \draw (cl) to [out=50,in=130,relative] 
 node[fill=white,inner sep=1,pos=0.5]{$6$} (cr);
 \draw (d) .. controls (-30:3.5) and (70:4) .. 
 node[fill=white,inner sep=1,pos=0.55]{$7$} (cl);
 \fill(d) circle (1mm); \fill(cd) circle (1mm); \fill(cl) circle (1mm); \fill(cr) circle (1mm);
\end{tikzpicture}
\end{figure}
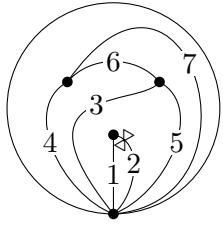
We consider flipping the marked surface in Figure \ref{Bsurface} at 2,3,4,5,7. The relative position of arc corresponding to $x_2$ and $x'_7$ is as in Figure \ref{twovariableposition}. 

\begin{figure}[ht]
 \centering
 \caption{Relative position of arc corresponding to $x_2$ and $x'_7$\label{twovariableposition}}
\begin{tikzpicture}[baseline=0mm,scale=0.7]
 \coordinate (d) at (0,-2);
 \coordinate (0) at (0,0);
 \draw (d) to node[pos=0.85]{\rotatebox{0}{\footnotesize $\bowtie$}} node[fill=white,inner sep=1,pos=0.55]{$x_2$} (0);
 \draw (0) .. controls (150:3.5) and (30:3.5) .. 
 node[fill=white,inner sep=1,pos=0.5]{$x'_7$} (0);
 \fill(d) circle (1mm); \fill(0) circle (1mm);
\end{tikzpicture}
\end{figure}
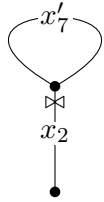
As mentioned in \cite{fst}*{Example 8.5}, considering the intersection number(\cite{fst}*{Definition 8.4}), we have $(x_2 \parallel x_7')_d=2,(x'_7 \parallel x_2)_d=1$. This example implies $(\cdot \parallel \cdot)_d$ does not satisfy the similar property of Proposition \ref{pr:f-compatibility-symmetry} (1),(2). 
\end{example} 

Next, we consider a generalization of Proposition \ref{classical-exchangeability}. The following statement is clear:
\begin{theorem}\label{f-exchangeability1}
For any cluster algebra $\Acal(B)$ and any pair of its cluster variables $x$ and $x'$, if there exists a set $X$ of cluster variables such that $X\cup x$ and $X\cup x'$ are both clusters, then $(x\parallel x')=(x'\parallel x)=1$.
\end{theorem}
\begin{proof}
We take a seed whose cluster is $X\cup x$ as the initial seed and consider a mutation such that it changes cluster from $X\cup x$ to $X\cup x'$. By Lemma \ref{pathconnect}, there is a mutation satisfying this condition. The statement is followed by definition of the cluster mutation \eqref{eq:x-mutation} and of the $f$-vectors (Definition \ref{def:f-vector}).
\end{proof}
The converse of Theorem \ref{f-exchangeability1} is still open:
\begin{conjecture}\label{f-exchangeability2}
For any cluster algebra $\Acal(B)$ and any pair of its cluster variables $x$ and $x'$, if $(x\parallel x')=(x'\parallel x)=1$, then there exists a set $X$ of cluster variables such that $X\cup x$ and $X\cup x'$ are both clusters.
\end{conjecture}

We call Theorem \ref{f-exchangeability1} and Conjecture \ref{f-exchangeability2} the \emph{exchangeability property}. In the case of finite type, Conjecture \ref{f-exchangeability2} is correct:

\begin{theorem}\label{f-exchangeability2finite}
For any cluster algebra $\Acal(B)$ of finite type and any pair of its cluster variables $x$ and $x'$, if $(x\parallel x')=(x'\parallel x)=1$, then there exists a set $X$ of cluster variables such that $X\cup x$ and $X\cup x'$ are both clusters.
\end{theorem}
\begin{proof}
It follows from Proposition \ref{classical-exchangeability} and Theorem \ref{classical-fcorrespondence}. 
\end{proof}
Also in the case of rank 2, we can prove Conjecture \ref{f-exchangeability2} by using descriptions of the $F$-matrices.
\begin{theorem}\label{f-exchangeability2rank2}
For any cluster algebra $\Acal$ of rank 2 and any pair of its cluster variables $x$ and $x'$, if $(x\parallel x')=(x'\parallel x)=1$, then there exists a set $X$ of cluster variables such that $X\cup x$ and $X\cup x'$ are both clusters.
\end{theorem}
\begin{proof}
 If $\Acal$ is of finite type, the result follows from Theorem \ref{f-exchangeability2finite}. We assume that $\Acal$ is not of finite type. Let $\Sigma=(\mathbf{x},\mathbf{y},B)$ be a labeled seed of $\Acal$ which contains the cluster variable $x$. Without loss of generality, we may assume that $B=\begin{bmatrix}0&b\\-c&0
\end{bmatrix}$ for some $b,c\in \ZZ_{\geq 0}$ such that $bc\geq 4$.

We name vertices of $\TT_2$ by the rule of \eqref{A2tree} and fix a cluster pattern $t_n\mapsto (\xx_{t_n},\yy_{t_n}, B_{t_n})$ by assigning the seed $\Sigma$ to the vertex $t_0$. We abbreviate $\xx_{t_n}$ (resp., $\yy_{t_n},\ B_{t_n},\Sigma_{t_n}$) to $\xx_n$ (resp., $\yy_{n},\ B_{n},\ \Sigma_n$). We also abbreviate cluster variables, $f$-vectors and $F$-matrices in the same way. 

Let us consider the case $x=x_{1;0}$, the case $x=x_{2;0}$ can be proved similarly. In this case, it suffices to prove that $x'\in \{x_{1;2},x_{1;-1}\}$.

A direct computation show that 
\begin{align}
 F_{0}^{B;t_0}=\begin{bmatrix}0&0\\0&0\end{bmatrix},\quad F_{1}^{B;t_0}=\begin{bmatrix}0&0\\0&1\end{bmatrix}. 
\end{align}
By descriptions of $D$-matrices of rank 2 \cite{llz}*{(1.13)} and Theorem \ref{f=d}, if $n>0$ is even, then we have
\begin{align}\label{f-des1}
F_n^{B;t_0}=\begin{bmatrix}
S_{\frac{n-2}{2}}(u)+S_{\frac{n-4}{2}}(u)& bS_{\frac{n-4}{2}}(u)\\
cS_{\frac{n-2}{2}}(u) &S_{\frac{n-2}{2}}(u)+S_{\frac{n-4}{2}}(u)
\end{bmatrix},
\end{align}
and if $n>1$ is odd, then we have
\begin{align}\label{f-des2}
F_n^{B;t_0}=\begin{bmatrix}
S_{\frac{n-3}{2}}(u)+S_{\frac{n-5}{2}}(u)& bS_{\frac{n-3}{2}}(u)\\
cS_{\frac{n-3}{2}}(u) &S_{\frac{n-1}{2}}(u)+S_{\frac{n-3}{2}}(u)
\end{bmatrix},
\end{align}
where $u=bc-2$ and $S_p(u)$ is a (normalized) Chebyshev polynomial of the second kind, that is, 
\begin{align}
S_{-1}(u)=0,\quad S_0(u)=1,\quad S_p(u)=uS_{p-1}(u)-S_{p-2}(u)\ (p\in\mathbb{N}).
\end{align}
When $n<0$, $F_n^{B;t_0}$ is the following matrix:
\begin{align}\label{negativen}
F_n^{B;t_0}=\begin{bmatrix}
f_{22;-n}^{-B^\top}& f_{21;-n}^{-B^\top}\\
f_{12;-n}^{-B^\top}& f_{11;-n}^{-B^\top}
\end{bmatrix},
\end{align}
where $f_{ij;-n}^{-B^\top}$ is the $(i,j)$ entry of $F_{-n}^{-B^\top;t_0}$. 
For any $p\geq 0$, we have $S_p(u)-S_{p-1}(u)>0$. Indeed, we have $S_0(u)-S_{-1}(u)=1>0$. Assume that $S_q(u)-S_{q-1}(u)>0$, then $S_{q+1}(u)-S_q(u)=(u-1)S_{q}(u)-S_{q-1}(u)>0$. In particular, for any $p\in \ZZ_{\geq -1}$, we have
\begin{align}\label{Sp}
\cdots>S_{p+1}(u)>S_{p}(u)>\cdots>S_{0}(u)=1>S_{-1}(u)=0. 
\end{align}

We claim that $(x_{1;0}\parallel x_{i;n})(x_{i;n}\parallel x_{1;0})>1$ for $i\in \{1,2\}$ whenever $n\geq 4$ or $n\leq -3$. As a consequence, $x'\neq x_{i;n}$ for any $n\geq 4$ and $n\leq -3$ and $x'\in \bigcup\limits_{-2\leq n\leq 3}\mathbf{x}_n$. 
Recall that according to Proposition \ref{pr:f-compatibility-symmetry} (2), $(x_{1;0}\parallel x_{i;n})=0$ if and only if $(x_{i;n}\parallel x_{1;0})=0$.
For $i=1$ and $n\geq 4$, 
\begin{align*}
  (x_{1;0}\parallel x_{1;n})(x_{1;n}\parallel x_{1;0})&=f_{11;n}^{B;t_0}(x_{1;n}\parallel x_{1;0}) \\
  &\geq (S_1(u)+S_0(u))(x_{1;n}\parallel x_{1;0}) & \text{by (\ref{f-des1})(\ref{f-des2})}\\
  &>1. &\text{by (\ref{Sp})}
\end{align*}
For $i=2$ and $n\geq 4$,
\begin{align*}
(x_{1;0}\parallel x_{2;n})(x_{2;n}\parallel x_{1;0})&=f_{12;n}^{B;t_0}(x_{2;n}\parallel x_{1;0})\\
&\geq b(x_{2;n}\parallel x_{1;0})& \text{by (\ref{f-des1})(\ref{f-des2})}\\
& =b(x_{1;0}^{\vee}\parallel x_{2;n}^{\vee})& \text{by Proposition \ref{pr:f-compatibility-symmetry}(1)}\\
&=bf_{12;n}^{-B^\top;t_0}\\
&\geq bc\geq 4. & \text{by (\ref{f-des1})(\ref{f-des2})}\\
\end{align*}
This completes the proof of the claim for $n\geq 4$. For $n\leq -3$, one can prove the statement by using (\ref{negativen}) similarly and we omit the proof.

According to the cluster pattern, we have
\[x_{2;-2}=x_{2;-3}, x_{1;-2}=x_{1;-1}, x_{2;-1}=x_{2;0}, x_{1;0}=x_{1;1}, x_{2;1}=x_{2;2}, x_{1;2}=x_{1;3}, x_{2;3}=x_{2;4}.
\]
By the above claim, we conclude that $x'\neq x_{2;-2}$ and $x'\neq x_{2;3}$.
By Theorem \ref{fcompatibilityproperty}, we have 
\[(x_{1;0}\parallel x_{1;0})=(x_{1;0}\parallel x_{2;0})=(x_{1;0}\parallel x_{2;1})=0
\]
since there is a cluster containing $x_{1;0}$ and $x_{i;n}$ for $(i,n)\in \{(1,0),(2,0),(2,1)\}$. It follows that $x'\in \{x_{1;-1},x_{1;2}\}$.
This finishes the proof.
\end{proof}

\begin{remark}
In the case of finite type or rank 2, the $d$-compatibility degree also has the exchangeability property. However, it is not correct in general. See Section \ref{ss:counterexam}.
\end{remark}
In the last section, we prove Conjecture \ref{f-exchangeability2} in some more cases.

\section{exchangeability property of $2$-Calabi-Yau categories}

Let $k$ be an algebraically closed field and $D=\Hom_k(-,k)$ the usual duality over $k$. Let $\mathcal{T}$ be a Krull-Schmidt category over $k$ and $M$ an object of $\mathcal{T}$. Denote by $\add M$ the full subcategory of $\mathcal{T}$ consisting of objects which are direct summands of finite direct sums of copies of $M$. Denote by $|M|$ the number of non-isomorphic indecomposable direct summands of $M$.

\subsection{Reminder on cluster-tilting theory in $2$-Calabi-Yau categories}\label{ss:cluster-tilting}
Let $\mathcal{C}$ be a $2$-Calabi-Yau triangulated category over $k$ with suspension functor $\Sigma$. In particular, for any $X,Y\in \mathcal{C}$, we have the following bifunctorially isomorphism
\[\Hom_{\mathcal{C}}(X,Y)\cong D\Hom_\mathcal{C}(Y,\Sigma^2 X).
\]

 An object $M\in \mathcal{C}$ is {\it rigid} if $\Hom_\mathcal{C}(M,\Sigma M)=0$. It is {\it maximal rigid} if it is maximal with respect to this property, \ie $\Hom_\mathcal{C}(M\oplus Y, \Sigma M\oplus \Sigma Y)=0$ implies that $Y\in \add M$.
 An object $T\in\mathcal{C}$ is a {\it cluster-tilting object} if $T$ is rigid and $\Hom_\mathcal{C}(T,\Sigma Y)=0$ implies that $Y\in \add T$. 
It is clear that cluster-tilting objects are maximal rigid objects, but the converse is not true in general. For $2$-Calabi-Yau categories with cluster-tilting objects, we have the following.
\begin{lemma}[\cite{ZZ}*{Theorem~2.6}]\label{l:zz}
Let $\mathcal{C}$ be a $2$-Calabi-Yau triangulated category with cluster-tilting objects. 
 Every maximal rigid object is a cluster-tilting object.
 \end{lemma}
 
 Let $T$ be a basic cluster-tilting object of $\mathcal{C}$. According to ~\cite{KeR}, for any $M\in \mathcal{C}$, we have a triangle
 \[T_1^M\to T_0^M\to M\to \Sigma T_1^M,
 \]
 where $T_1^M, T_0^M\in \add T$. The {\it index} of $M$ with respect to $T$ is defined to be the class in $\go(\add T)$:
 \[\operatorname{ind}_T(M)=[T_0^M]-[T_1^M],
 \]
 where $\go(\add T)$ is the split Grothendieck group of $\add T$ and $[\ast]$ stands for the image of $\ast$ in $\go(\add T)$.
The following has been proved in~\cite{DK}.
 
 \begin{lemma}[\cite{DK}*{Theorem 2.6}]\label{l:DK}
 Let $\mathcal{C}$ be a $2$-Calabi-Yau triangulated category with cluster-tilting objects. 
 For any basic cluster-tilting object $T$ and $T'$ of $\mathcal{C}$, we have $|T|=|T'|$.
\end{lemma}

Let $T=M\oplus \overline{T}$ be a basic cluster-tilting object with indecomposable direct summand $M$. It was shown in~\cite{IY} that there exists a unique object $M^*\not\cong M$ such that $\mu_M(T):=M^*\oplus \overline{T}$ is again a cluster-tilting object of $\mathcal{C}$. Moreover, $M^*$ is uniquely determined by the {\it exchange triangles}
\[M\xrightarrow{f}B\xrightarrow{g}M^*\to \Sigma M~\text{and}~M^*\xrightarrow{f'}B'\xrightarrow{g'}M\to \Sigma M^*,
\]
where $g, g'$ are minimal right $\add \overline{T}$-approximations.
The cluster-tilting object $\mu_M(T)$ is called the {\it mutation of $T$ at the indecomposable direct summand $M$}. In this case, $(M, M^*)$ is called an {\it exchange pair}.
A basic cluster-tilting object $T'$ is {\it reachable from} $T$ if $T'$ is obtained from $T$ by a finite sequence of mutations. A rigid object $M$ is {\it reachable from} $T$ if it lies in $\add T'$ for a cluster-tilting object $T'$ reachable from $T$.
The {\it cluster-tilting graph} or {\it exchange graph} of $\mathcal{C}$ has as vertices the isomorphism classes of basic cluster-tilting objects of $\mathcal{C}$, while two vertices $T$ and $T'$ are connected by an edge if and only if $T'$ is a mutation of $T$. If the exchange graph of $\mathcal{C}$ is connected, then every rigid object is reachable from a given cluster-tilting object $T$.

Let $\mathcal{C}$ be a $2$-Calabi-Yau triangulated category with cluster-tilting objects. 
For a given rigid object $U\in \mathcal{C}$, we define the full subcategory $\mathcal{Z}(U)$ of $\mathcal{C}$ as follows
\[\mathcal{Z}(\mathcal{U}):=\{X\in \mathcal{C}~|~\Hom_\mathcal{C}(U, \Sigma X)=0\}\subseteq \mathcal{C}.
\]
By definition, we have $U\in \mathcal{Z}(\mathcal{U})$ and hence we can form the additive quotient \[\mathcal{C}_{U}:=\mathcal{Z}(\mathcal{U})/[\add U].\] 
It has the same objects as $\mathcal{Z}(U)$ and for $X, Y\in \mathcal{Z}(\mathcal{U})$ we have
\[\Hom_{\mathcal{C}_{U}}(X,Y):=\Hom_\mathcal{C}(X,Y)/ [\add U](X,Y),
\]
where $[\add U](X,Y)$ denotes the subgroup of $\Hom_\mathcal{C}(X,Y)$ consisting of morphisms which factor through objects in $\add U$. 

For $X\in \mathcal{Z}(U)$, let $X\to U_X$ be a minimal left $\add U$-approximation. We define $X\langle 1\rangle$ to be the cone
\[X\to U_X\to X\langle 1\rangle \to \Sigma X.
\]
\begin{theorem}[\cite{IY}*{Theorems 4.7, 4.9}]\label{t:IY-reduction}
The category $\mathcal{C}_{U}$ is a $2$-Calabi-Yau triangulated category with suspension functor $\langle 1\rangle$. Moreover, there is a bijection between the set of cluster-tilting objects of $\mathcal{C}$ containing $U$ as a direct summand and the set of cluster-tilting objects of $\mathcal{C}_{U}$.
\end{theorem}
We refer to $\mathcal{C}_U$ the {\it Iyama-Yoshino's reduction} of $\mathcal{C}$ with respect to $U$.

\subsection{Reminder on categorification via $2$-Calabi-Yau categories}\label{ss:cluster-character}

For a quiver $Q$ without loops nor $2$-cycles, denote by $Q_0=\{1,\dots,n\}$ its vertex set. We define a skew-symmetric matrix $B(Q)=(b_{ij})\in M_n(\mathbb{Z})$ associated to $Q$, where
\[b_{ij}=\begin{cases}0& i=j;\\ |\{\text{arrows~ $i\to j$}\}|-|\{\text{arrows $j\to i$}\}|& i\neq j.\end{cases}
\]
On the other hand, for a given skew-symmetric integer matrix $B$, one can construct a quiver $Q$ without loops nor $2$-cycles such that $B=B(Q)$.

Let $\mathcal{C}$ be a $2$-Calabi-Yau triangulated category with cluster-tilting objects.
The category $\mathcal{C}$ has {\it no loops nor $2$-cycles} provided that for each basic cluster-tilting object $T$, the Gabriel quiver $Q_T$ of its endomorphism algebra $\End_\mathcal{C}(T)$ has no loops nor $2$-cycles.
The category $\mathcal{C}$ admits a {\it cluster structure}~\cite{BIRS} if 
\begin{itemize}
\item for each basic cluster-tilting object $T$, the Gabriel quiver $Q_T$ of its endomorphism algebra $\End_\mathcal{C}(T)$ has no loops and no $2$-cycles;
\item if $T=M\oplus \overline{T}$ is a basic cluster-tilting object with $M$ indecomposable, then $B(Q_{\mu_M(T)})$ is the Fomin-Zelevinsky's mutation~(\ref{eq:matrix-mutation}) of $B(Q_T)$ at the vertex corresponding to $M$.
\end{itemize} 
According to~\cite{BIRS}*{Theorem II.1.6}, if $\mathcal{C}$ has no loops nor $2$-cycles, then $\mathcal{C}$ has a cluster structure.

Let $\mathcal{C}$ be a $2$-Calabi-Yau triangulated category with cluster-tilting objects. Let $T=\bigoplus_{i=1}^nT_i$ be a basic cluster-tilting object of $\mathcal{C}$ with indecomposable direct summands $T_1,\dots, T_n$. For any $M\in \mathcal{C}$, the {\it cluster character}~\cite{Pa} ${CC}_T(M)$ of $M$ with respect to $T$ is defined as follows:
\[ {CC}_T(M):=\mathbf{x}^{\operatorname{ind}_T(M)}\sum_{e}\chi(\operatorname{Gr}_e(\Hom_\mathcal{C}(T,\Sigma M)))\mathbf{x}^{B(Q_T)e}\in \mathbb{Z}[x_1^{\pm},\dots, x_n^{\pm}],
\]
where
\begin{enumerate}
\item[$\bullet$] for a vector $a=(a_1,\dots,a_n)^{\top}\in\mathbb{Z}^n$ or $\alpha=\sum_{i=1}^na_i[T_i]\in \go(\add T)$, we write $\mathbf{x}^\alpha=x_1^{a_1}\cdots x_n^{a_n}$;
\item[$\bullet$] $\Hom_\mathcal{C}(T, \Sigma M)$ is a right $\End_\mathcal{C}(T)$-module and $\operatorname{Gr}_e(\Hom_\mathcal{C}(T,\Sigma M))$ is the quiver Grassmanian of $\Hom_\mathcal{C}(T,\Sigma M)$ consisting of submodules with dimension vector $e$, which is a projective variety.
\item[$\bullet$] $\chi$ is the Euler-Poincar\'{e} characteristic.
\end{enumerate}

\begin{proposition}\label{p:cc-map}
Let $\mathcal{C}$ be a $2$-Calabi-Yau triangulated category with a cluster-tilting object $T$. Assume that $\mathcal{C}$ admits a cluster structure. Denote by $\mathcal{A}_T:=\mathcal{A}(B(Q_T))$ the cluster algebra associated to $T$. 
\begin{itemize}
\item[$(a)$]The map $M\mapsto {CC}_T(M)$ induces a bijection between the set of indecomposable rigid objects reachable from $T$ and the set of cluster variables of the cluster algebra $\mathcal{A}_T$.
\item [$(b)$] The bijection of $(a)$ induces a bijection between the set of basic cluster-tilting objects reachable from $T$ and the set of clusters of $\mathcal{A}_T$. Furthermore, the bijection is compatible with mutations.
\end{itemize}
\end{proposition}
\begin{proof}
The surjectivity of $M\mapsto {CC}_T(M)$ follows from~\cite{Pa}, and we refer to~\cite{fk}*{Proposition 2.3} for a proof. The injectivity of the map has been proved by~\cite{CKLP}*{Corollary 3.5} in a general setting.

\end{proof}
\begin{remark}\label{r:categorification}
In the above proposition, in order to categorify the cluster algebra $\mathcal{A}_T$, the condition that $\mathcal{C}$ admits a cluster structure can be weakened. Namely, it suffices to assume that each basic cluster-tilting object which is reachable from $T$ satisfies the conditions in the definition of cluster structure.
\end{remark}
 
\begin{theorem}\label{p:categorical-f-degree}
Let $\mathcal{C}$ be a $2$-Calabi-Yau triangulated category which admits a cluster structure. Let $T$ be a basic cluster-tilting object of $\mathcal{C}$ such that $B(Q_T)$ is of full rank.
For any indecomposable rigid objects $L,N$ reachable from $T$, we have
\[({CC}_T(L)\parallel {CC}_T(N))=\dim_k\Hom_\mathcal{C}(L, \Sigma N).
\]
\end{theorem}

\begin{proof}
Denote by $|T|=n$.
We fix a {\it cluster-tilting pattern} for cluster-tilting objects which are reachable from $T$. Namely, for each vertex $t\in \mathbb{T}_n$, we assign $t$ a basic cluster-tilting object $T_t=\bigoplus_{i=1}^n T_{i,t}$ reachable from $T$ such that the basic cluster-tilting objects assigned to $t$ and $t'$ linked by an edge labeled $j$ are obtained from each other by mutation at the $j$th indecomposable direct summand. For each $t\in \mathbb{T}_t$, we denote by $B_t:=B(Q_{T_t})$, which is the exchange matrix of the labeled cluster $\mathbf{x}_t:=({CC}_T(T_{1,t}),\dots, {CC}_T(T_{n,t}))$. We abbreviate $B_{t_0}=B$. 

Let $t_0$ and $s$ be two vertices of $\mathbb{T}_n$ such that $L\in \add T_{t_0}$ and $N\in \add T_s$. Without loss of generality, we may assume that $N\cong T_{1,s}$. 
We consider the cluster algebra $\mathcal{A}_\bullet(B)$ with principal coefficients at the rooted vertex $t_0$ and set 
 $\mathbf{x}_{t_0}:=(x_1,\dots, x_n)$ and $\mathbf{y}_{t_0}:=(y_1,\dots, y_n)$.

Let $z:=z_{1,s}^{B;t_0}$ be the first cluster variable of $\mathcal{A}_\bullet(B)$ associated to the vertex $s\in \mathbb{T}_n$. Let $F_z(\mathbf{y}):=F_{1,s}^{B;t_0}(\mathbf{y})=\sum_{v\in \mathbb{N}^n}c_vy^v\in \mathbb{Z}[y_1,\cdots, y_n]$ be the $F$-polynomial of the cluster variable $z$, where $c_v\in \mathbb{Z}$. By definition, it suffices to prove that 
\[F_z(\mathbf{y})=\sum_e\chi(\opname{Gr}_e(\Hom_\mathcal{C}(T_{t_0}, \Sigma N)))\mathbf{y}^e.
\]
Denote by $g_z\in \mathbb{Z}^n$ the $g$-vector of $z$. By the separation formula \eqref{eq:xjt=F/F}, we have
\[z=\mathbf{x}^{g_z}F_z(\hat{y}_1, \dots, \hat{y}_n),
\]
where $\hat{y}_i=y_i\prod_{j=1}^nx_j^{b_{ji;t_0}}$ for $1\leq i\leq n$.
Note that we also have 
\[{CC}_{T_{t_0}}(N)=z|_{y_1=\cdots =y_n=1}.
\]
Consequently, 
\[\mathbf{x}^{\opname{ind}_{T_{t_0}}(N)}\sum_{e}\chi(\opname{Gr}_e(\Hom_\mathcal{C}(T_{t_0}, \Sigma N)))\mathbf{x}^{Be}=\mathbf{x}^{g_z}\sum_vc_v\mathbf{x}^{Bv}.
\]
It follows from \cite{DK}*{Section 4} and \cite{dwz}*{Theorem 1.7} that $g_z=\opname{ind}_{T_{t_0}}(N)$ by identifying $\mathbb{Z}^n$ with $ \go(\add T_{t_0})$. Hence we have
\[\sum_{e}\chi(\opname{Gr}_e(\Hom_\mathcal{C}(T_{t_0}, \Sigma N)))\mathbf{x}^{Be}=\sum_vc_v\mathbf{x}^{Bv}.
\]
As note that the rank of an exchange matrix is invariant under mutations, by the condition that $B(Q_T)$ is of full rank, we know that $\opname{rank} B=n$. Therefore, $Be\neq Bf$ whenever $e\neq f\in \mathbb{Z}^n$. It follows that 
\[F_{1,s}^{B;t_0}(\mathbf{y})=F_z(\mathbf{y})=\sum_e\chi(\opname{Gr}_e(\Hom_\mathcal{C}(T_{t_0}, \Sigma N)))\mathbf{y}^e.
\]
\end{proof}

\subsection{The exchangeability property of $2$-Calabi-Yau categories}\label{ss:exchangeable-property-2-CY}
The aim of this subsection is to establish the exchangeability property of $2$-Calabi-Yau categories. We begin with the following special case.

\begin{lemma}\label{l:special}
Let $\mathcal{C}$ be a $2$-Calabi-Yau triangulated categories with cluster-tilting objects. Let $X$ be an indecomposable rigid object such that $\Hom_\mathcal{C}(X,X)\cong k$ and $X\cong\Sigma^2 X$. Then $(X, \Sigma X)$ is an exchange pair of $\mathcal{C}$.
\end{lemma}
\begin{proof}
Since $\mathcal{C}$ is $2$-Calabi-Yau, we know that the suspension functor $\Sigma$ coincides with the Auslander-Reiten translation of $\mathcal{C}$. By the condition that $\Hom_\mathcal{C}(X,X)\cong k$ and $X\cong \Sigma^2 X$, we conclude that 
\[\Sigma X\xrightarrow{f} 0\to X\to \Sigma^2X\cong X
\]
 is the Auslander-Reiten triangle ending at $X$. Let $\overline{T}$ be an object such that $\overline{T}\oplus X$ is a basic cluster-tilting object of $\mathcal{C}$. In particular, we have $X\not\in \add \overline{T}$. Such a $\overline{T}$ exists by Lemma~\ref{l:zz}, namely, $X$ can be complemented to a basic maximal rigid object. By definition of AR triangles, each morphism from $\Sigma X$ to $\Sigma\overline{T}$ factors through the morphism $f$. In particular, $\Hom_\mathcal{C}(\Sigma X,\Sigma\overline{T})=0$. 
 Consequently, $\Sigma X\oplus \overline{T}$ is rigid. 
 By $X\cong \Sigma^2X$ and $\Hom_\mathcal{C}(X, \Sigma \overline{T})=0$, we deduce that $\Sigma X\not\in \add \overline{T}$. It follows that 
 $|X\oplus \overline{T}|=|\Sigma X\oplus \overline{T}|$, we conclude that $\Sigma X\oplus \overline{T}$ is a cluster-tilting object by Lemma~\ref{l:zz} and Lemma~\ref{l:DK}.
\end{proof}
The following is the main result of this subsection, which generalizes~\cite{bmrrt}*{Theorem 7.5} for cluster categories associated with hereditary algebras.
\begin{theorem}\label{t:exchangeable-2-CY}
Let $\mathcal{C}$ be a $2$-Calabi-Yau triangulated category with cluster-tilting objects.
Let $L, N$ be two indecomposable rigid objects of $\mathcal{C}$.
 If $\dim_k\Hom_{\mathcal{C}}(L, \Sigma N)=1$, then $(L, N)$ is an exchange pair. 
 Assume moreover that $\mathcal{C}$ admits a cluster structure, then $\dim_k\Hom_{\mathcal{C}}(L, \Sigma N)=1$ if and only if $(L, N)$ is an exchange pair.
\end{theorem}
\begin{proof}Let $h\in \Hom_\mathcal{C}(N, \Sigma L)$ and $t\in \Hom_\mathcal{C}(L,\Sigma N)$ be non-zero morphisms.
We consider the triangles determined by $h$ and $t$: 
\begin{align}\label{f:5.1}
L\xrightarrow{f} M\xrightarrow{g} N\xrightarrow{h} \Sigma L 
\end{align}
~and 
\begin{align}\label{f:5.2}
N\xrightarrow{r} M'\xrightarrow{s}L\xrightarrow{t} \Sigma N.
\end{align}
If both $h$ and $t$ are isomorphisms, then we are in the situation of Lemma~\ref{l:special} and we are done.
Now suppose that at least one of $h$ and $t$ is not an isomorphism. In particular, at least one of $M$ and $M'$ is non-zero.

 Applying $\Hom_\mathcal{C}(-,\Sigma L)$ to the triangle ~(\ref{f:5.1}), we obtain a long exact sequence
\[\Hom_{\mathcal{C}}(\Sigma L, \Sigma L)\xrightarrow{h^*}\Hom_{\mathcal{C}}(N,\Sigma L)\to \Hom_{\mathcal{C}}(M,\Sigma L)\to \Hom_{\mathcal{C}}(L,\Sigma L).
\]
 As note that $h^\ast(\id_{\Sigma L})=h\neq 0$, hence $h^*$ is non-zero.
Since $\dim_k\Hom_\mathcal{C}(N,\Sigma L)=1$, we conclude that $h^*$ is surjective. On the other hand, by $\Hom_{\mathcal{C}}(L,\Sigma L)=0$, we deduce that $\Hom_{\mathcal{C}}(M,\Sigma L)=0$. By the $2$-Calabi-Yau duality, we obtain
\[\Hom_{\mathcal{C}}(M,\Sigma L)=0=\Hom_{\mathcal{C}}(L,\Sigma M).
\]
Similarly, by applying the functor $\Hom_\mathcal{C}(\Sigma^{-1}N,-)$ to the triangle ~(\ref{f:5.1}), one can show that 
\[\Hom_\mathcal{C}(N,\Sigma M)=0=\Hom_{\mathcal{C}}(M,\Sigma N).
\]
Now applying the functor $\Hom_{\mathcal{C}}(-,\Sigma M)$ to the triangle ~(\ref{f:5.1}), we have 
\[\Hom_\mathcal{C}(M,\Sigma M)=0.
\]
In particular, we have proved that $L\oplus M$ and $N\oplus M$ are rigid objects of $\mathcal{C}$.
Similarly, one can show that $L\oplus M'$ and $N\oplus M'$ are rigid. 

By applying $\Hom_\mathcal{C}(-,\Sigma M')$ to the triangle ~(\ref{f:5.1}), we obtain
\[\Hom_\mathcal{C}(N, \Sigma M')\to \Hom_\mathcal{C}(M,\Sigma M')\to \Hom_\mathcal{C}(L,\Sigma M').
\]
In particular, we have $\Hom_\mathcal{C}(M,\Sigma M')=0$. By the $2$-Calabi-Yau duality, we obtain $\Hom_\mathcal{C}(M',\Sigma M)=0$.

Let $R$ be a basic rigid object such that $\add R=\add M\oplus M'$.
In particular, $R\oplus L$ and $R\oplus N$ are rigid objects. It follows that $L, N\not\in \add R$. Indeed, if $L\in \add R$, then by $\Hom_\mathcal{C}(N, \Sigma R)=0$, we deduce that $\Hom_\mathcal{C}(N, \Sigma L)=0$, a contradiction.
It remains to show that there exists an object $P\in \mathcal{C}$ such that $P\oplus R\oplus L$ and $P\oplus R\oplus N$ are cluster-tilting objects.

We consider the Iyama-Yoshino's reduction $\mathcal{C}_R$ of $\mathcal{C}$ with respect to $R$. 
 It is routine to check that
\begin{itemize}
\item $f$ and $r$ are minimal left $\add R$--approximations of $L$ and $N$ respectively;
\item $g$ and $s$ are minimal right $\add R$--approximations of $N$ and $L$ respectively.
\end{itemize}

Consequently, we have $L\cong N\langle 1\rangle$ and $ N\cong L\langle 1\rangle$ in $\mathcal{C}_R$.
Furthermore, $L\cong L\langle 2\rangle$ in $\mathcal{C}_R$.
Applying the functor $\Hom_\mathcal{C}(L,-)$ to the triangle ~(\ref{f:5.2}), we obtain a long exact sequence
\[\Hom_\mathcal{C}(L,N)\to \Hom_\mathcal{C}(L, M')\xrightarrow{s^*} \Hom_\mathcal{C}(L, L)\xrightarrow{t^*}\Hom_\mathcal{C}(L,\Sigma N)\cong k.
\]
Since $s$ is a right $\add R$-approximation of $L$, we deduce that
\[\Hom_{\mathcal{C}_R}(L, N\langle 1\rangle)=\Hom_{\mathcal{C}_R}(L,L)\cong \Hom_\mathcal{C}(L,L)/\im s^*\cong k.
\]

Applying Lemma~\ref{l:special} to the $2$-Calabi-Yau category $\mathcal{C}_R$ and the object $L$, we conclude that there is an object $P\in \mathcal{C}_R$ such that $L\oplus P$ and $L\langle 1\rangle \oplus P=N\oplus P$ are cluster-tilting objects of $\mathcal{C}_R$.
Now by applying Theorem~\ref{t:IY-reduction}, we conclude that $P\oplus L\oplus R$ and $P\oplus N\oplus R$ are cluster-tilting objects of $\mathcal{C}$. This completes the proof of the first statement.

Now assume moreover that $\mathcal{C}$ admits a cluster structure. It remains to show that if $(L,N)$ is an exchange pair, then $\Hom_\mathcal{C}(L,\Sigma N)\cong k$. Let $\overline{T}$ be a rigid object such that $\overline{T}\oplus L$ and $\overline{T}\oplus N$ are cluster-tilting objects. Denote by 
\[N\to U\xrightarrow{v} L\to \Sigma N
\]
one of the exchange triangles associated with $L$ and $N$.
 By definition of cluster structure, we know that $\mathcal{C}$ has no loops. In particular, each non-isomorphism $u:L\to L$ factors through $v$. Applying $\Hom_\mathcal{C}(L,-)$ to the triangle yields that 
 \[\Hom_\mathcal{C}(L,\Sigma N)\cong \Hom_\mathcal{C}(L,L)/\im v^*\cong k,
 \]
 where the last isomorphism follows from the fact that $\Hom_\mathcal{C}(L,L)$ is a local algebra.
\end{proof}

\section{exchangeability property for three classes of cluster algebras}
In this section, we prove Conjecture~\ref{f-exchangeability2} for
\begin{enumerate}
 \item[$\bullet$] acyclic cluster algebras of skew-symmetric type;
 \item[$\bullet$] cluster algebras arising from weighted projective lines;
 \item[$\bullet$] cluster algebras arising from marked surfaces. 
\end{enumerate}
Our approach depends on the existence of additive categorification by $2$-Calabi-Yau categories.

\subsection{Reminder on cluster categories associated with hereditary categories}

Let $\mathcal{H}$ be a hereditary abelian category with tilting objects. Let $\mathcal{D}^b(\mathcal{H})$ be the bounded derived category of $\mathcal{H}$. Denote by $\tau:\mathcal{D}^b(\mathcal{H})\to \mathcal{D}^b(\mathcal{H}) $ the Auslander-Reiten translation and $\Sigma$ the suspension functor of $\mathcal{D}^b(\mathcal{H})$. The {\it cluster category} $\mathcal{C}_{\mathcal{H}}$ of $\mathcal{H}$ is defined as the orbit category $\mathcal{D}^b(\mathcal{H})/\tau^{-1}\circ \Sigma$ of $\mathcal{D}^b(\mathcal{H})$~\cite{bmrrt} ({\it cf.} ~\cite{ccs} for the $\mathbb{A}_n$ case). Keller~\cite{ke} proved that $\mathcal{C}_{\mathcal{H}}$ admits a canonical triangle structure such that the projection functor $\pi:\mathcal{D}^b(\mathcal{H})\to \mathcal{C}_{\mathcal{H}}$ is a triangle functor. Moreover, $\mathcal{C}_{\mathcal{H}}$ is a $2$-Calabi-Yau triangulated category. Let $T$ be a tilting object of $\mathcal{H}$, it was proved in~\cite{bmrrt}*{Theorem 3.3 and Proposition 3.4} that $\pi(T)$ is a cluster-tilting object of $\mathcal{C}_{\mathcal{H}}$.

Let $Q$ be a finite acyclic quiver and $kQ$ its path algebra. The category $\opname{mod} kQ$ of finitely generated right $kQ$-modules is a hereditary abelian category with tilting objects. In particular, the free module $kQ$ is a tilting object. Another important class of hereditary abelian categories with tilting objects comes from weighted projective lines introduced by Geigle and Lenzing~\cite{gl}.
According to Happel's classification theorem~\cite{ha}, each connected hereditary abelian category with tilting object is derived equivalent to $\opname{mod} kQ$ for the path algebra of an acyclic quiver $Q$ or to the category $\opname{coh} \mathbb{X}$ of coherent sheaves over a weighted projective line $\mathbb{X}$. We remark that if $\mathcal{H}$ and $\mathcal{H}'$ are hereditary abelian categories with tilting objects such that $\mathcal{D}^b(\mathcal{H})\cong \mathcal{D}^b(\mathcal{H'})$, then $\mathcal{C}_{\mathcal{H}}\cong \mathcal{C}_{\mathcal{H}'}$. By Happel's result, we may unify ~\cite{bmr08}*{Theorem 1.3 and Proposition 3.2} and ~\cite{bkl}*{Theorem 3.1} as follows.
\begin{theorem}\label{t:cluster-cat}
Let $\mathcal{H}$ be a hereditary abelian category with tilting objects. The cluster category $\mathcal{C}_{\mathcal{H}}$ admits a cluster structure.
\end{theorem}

As a consequence, for each basic cluster-tilting object $T$ of $\mathcal{C}_\mathcal{H}$, we have a skew-symmetric integer matrix $B(Q_T)$. The cluster algebra $\mathcal{A}(B(Q_T))$ is an acyclic cluster algebra if $\mathcal{H}=\opname{mod} kQ$ for an acyclic quiver $Q$. If $Q'$ is mutation equivalent to an acyclic quiver $Q$, then there exists a cluster-tilting object $T$ in $\mathcal{C}_{\mathcal{H}}$ such that $Q'=Q_T$. Therefore, if the initial matrix $B(Q')$ is obtained from $Q'$, then $\Acal(B(Q'))$ is an acyclic cluster algebra. We refer to $\mathcal{A}(B(Q_T))$ a {\it cluster algebra arising from weighted projective line} whenever $\mathcal{H}=\opname{coh} \mathbb{X}$ for a weighted projective line $\mathbb{X}$.

The connectedness of cluster-tilting graph of $\mathcal{C}_{\mathcal{H}}$ plays a key role in our approach to the exchangeability property of cluster algebras arising from $\mathcal{C}_{\mathcal{H}}$.
\begin{theorem}[\cite{fugeng}*{Theorem 1.2}]\label{t:connected-exchange-graph}

Let $\mathcal{H}$ be a hereditary abelian category with tilting objects. The cluster-tilting graph of $\mathcal{C}_{\mathcal{H}}$ is connected.
\end{theorem}

\subsection{Acyclic cluster algebras of skew-symmetric type}\label{ss:ex-acylic}

This subsection aims to establish Conjecture~\ref{f-exchangeability2} for acyclic cluster algebras of skew-symmetric type.
\begin{theorem}\label{t:conj-acyclic}
Conjecture~\ref{f-exchangeability2} holds for acyclic cluster algebras of skew-symmetric type.
\end{theorem}
\begin{proof}
According to \cite{CKLP}*{Corollary 5.5}, the exchange graph of a skew-symmetric cluster algebra only depends on its initial exchange matrix. On the other hand, the definition of $f$-vectors does not depend on the coefficients. It suffices to prove this statement for acyclic skew-symmetric cluster algebras with principal coefficients. Furthermore, since a compatibility degree is not depend on the initial matrix (or quiver), we can assume the initial quiver $Q$ is acyclic. 

Let $Q$ be an acyclic quiver and $\mathcal{A}_\bullet(B(Q))$ the cluster algebra with principal coefficients. We are going to show that Conjecture~\ref{f-exchangeability2} is true for $\mathcal{A}_\bullet(B(Q))$. Let $Q_0=\{1,\dots,n\}$ be the vertex set of $Q$ and $Q_1$ the arrow set of $Q$. We introduce a new acyclic quiver $\hat{Q}$ as follows
\begin{itemize}
\item the set of vertices $\hat{Q}_0:=Q_0\sqcup \{i^*~|~\forall i\in Q_0\}$;
\item the set of arrows $\hat{Q}_1:=Q_1\sqcup\{i^*\to i~|~\forall i\in Q_0\}$.
\end{itemize}
In particular, $\mathcal{A}_\bullet(B(Q))$ is a subalgebra of the cluster algebra $\mathcal{A}(B(\hat{Q}))$ with trivial coefficients. We fix a cluster pattern of $\mathcal{A}(B(\hat{Q}))$ (resp. $\mathcal{A}_\bullet(B(Q))$) by assigning the initial seed to the rooted vertex $t_0\in \mathbb{T}_{2n}$ (resp. $t_0'\in \mathbb{T}_n$). By identifying the vertex $t_0'$ with $t_0$, the cluster pattern of $\mathcal{A}_\bullet(B(Q))$ can be identified with a subgraph of $\mathbb{T}_{2n}$.
Let $\mathcal{C}_{\hat{Q}}$ be the cluster category associated to $\opname{mod} k\hat{Q}$. For each vertex $i\in \hat{Q}_0$, denote by $P_i$ the indecomposable projective $k\hat{Q}$-modules associated to $i$. Set $T=\bigoplus_{i\in \hat{Q}_0} \Sigma P_i$. It is known that $T$ is a basic cluster-tilting object of $\mathcal{C}_{\hat{Q}}$. According to Theorem~\ref{t:cluster-cat}, Theorem~\ref{t:connected-exchange-graph} and Proposition~\ref{p:cc-map}, the cluster character ${CC}_T(?)$ yields a bijection between the set of indecomposable rigid objects of $\mathcal{C}_{\hat{Q}}$ and the set of cluster variables of $\mathcal{A}(B(\hat{Q}))$. 
Denote by $U:=\bigoplus_{i\in Q_0} \Sigma P_{i^*}$ and define
\[\mathcal{Z}(U):=\{X\in \mathcal{C}_{\hat{Q}}~|~\Hom_{\mathcal{C}_{\hat{Q}}}(U, \Sigma X)=0\}.\]
Let $L$ be an indecomposable rigid object of $\mathcal{C}_{\hat{Q}}$. The cluster variable ${CC}_T(L)$ belongs to $\mathcal{A}_\bullet(B(Q))$ if and only if $L\in \mathcal{Z}(U)$. 

Now let $L,N\in \mathcal{Z}(U)$ be indecomposable rigid objects of $\mathcal{C}_{\hat{Q}}$ such that \[({CC}_T(L)\parallel{CC}_T(N))_{\mathcal{A}_\bullet(B(Q))}=({CC}_T(N)\parallel{CC}_T(L))_{\mathcal{A}_\bullet(B(Q))}=1,\]
where the subscript is to indicate the compatibility degree of $\mathcal{A}_\bullet(B(Q))$. By Proposition~\ref{pr:f-compatibility-symmetry}~(3), we have
\[({CC}_T(L)\parallel{CC}_T(N))_{\mathcal{A}(B(\hat{Q}))}=({CC}_T(N)\parallel{CC}_T(L))_{\mathcal{A}(B(\hat{Q}))}=1.\]
Note that, as the exchange matrix $B(\hat{Q})$ is of full rank, we have $\dim_k\Hom_{\mathcal{C}_{\hat{Q}}}(L,\Sigma N)=1$ by Theorem~\ref{p:categorical-f-degree}.
Denote by $\mathcal{C}=\mathcal{Z}(U)/[\add U]$ the Iyama-Yoshino's reduction of $\mathcal{C}_{\hat{Q}}$ with respect to $U$. In particular, $L,N$ belong to $\mathcal{C}$. 
Let \[N\xrightarrow{f} B\to N\langle 1\rangle \to \Sigma N\] be a triangle of $\mathcal{C}_{\hat{Q}}$, where $f$ is a minimal left $\add U$-approximation. By applying $\Hom_{\mathcal{C}_{\hat{Q}}}(L,-)$ to the triangle yields a long exact sequence
\[\Hom_{\mathcal{C}_{\hat{Q}}}(L,N)\xrightarrow{f^*}\Hom_{\mathcal{C}_{\hat{Q}}}(L,B)\to \Hom_{\mathcal{C}_{\hat{Q}}}(L,N\langle 1\rangle)\to \Hom_{\mathcal{C}_{\hat{Q}}}(L,\Sigma N)\to \Hom_{\mathcal{C}_{\hat{Q}}}(L,\Sigma B)\cong 0.
\]
By definition of $\mathcal{C}$, we obtain that \[\Hom_\mathcal{C}(L, N\langle1\rangle)\cong \Hom_{\mathcal{C}_{\hat{Q}}}(L,\Sigma N).\] Hence \[\dim_k\Hom_\mathcal{C}(L, N\langle1\rangle)=1.\]
By Theorem~\ref{t:exchangeable-2-CY}, $(L,N)$ is an exchange pair of $\mathcal{C}$.
In other words, there is a rigid object $M\in \mathcal{C}$ such that $L\oplus M$ and $N\oplus M$ are cluster-tilting objects of $\mathcal{C}$. According to Theorem~\ref{t:IY-reduction}, $L\oplus M\oplus U$ and $N\oplus M\oplus U$ are cluster-tilting objects of $\mathcal{C}_{\hat{Q}}$. Now the statement follows from Proposition~\ref{p:cc-map}.
\end{proof}

\subsection{Cluster algebras arising from weighted projective lines}\label{ss:ex-wpl}

Fix a positive integer $t\geq 2$.
A {\it weighted projective line} $\mathbb{X}=\mathbb{X}(\mathbf{p},\boldsymbol{\lambda})$ over $k$ is given by a weight sequence $\mathbf{p}=(p_1,\dots,p_t)$
of positive integers, and 
a parameter sequence $\boldsymbol{\lambda}=(\lambda_1,\dots,\lambda_t)$ of pairwise distinct points of the projective line $\mathbb{P}_1(k)$. We always assume that $p_i\geq 2$ for all $1\leq i\leq t$.
Denote by $\opname{coh}\mathbb{X}$ the category of coherent sheaves over $\mathbb{X}$. We refer to~\cite{gl} for the construction of $\opname{coh}\mathbb{X}$.
It has been proved in~\cite{gl} that $\opname{coh}\mathbb{X}$ is a hereditary abelian category with tilting objects. In this case, we denote by $\mathcal{C}_{\mathbb{X}}$ the cluster category associated with $\opname{coh}\mathbb{X}$.

For a weight sequence $\mathbf{p}=(p_1,\cdots, p_t)$, we introduce a quiver $Q_{\mathbf{p}}$ in Figure~\ref{f:1}.
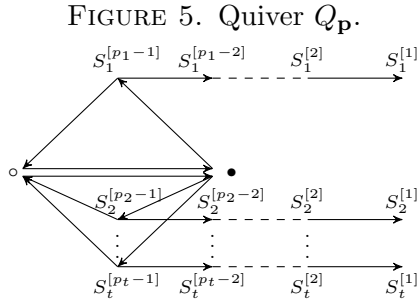
\begin{figure}[ht]

\caption{Quiver $Q_{\mathbf{p}}$.}\label{f:1}

\begin{tikzpicture}[
 scale=2.5,axis/.style={very thick, ->, >=stealth'},
 arrow/.style={->, >=stealth'},
 important line/.style={thick},
 dashed line/.style={dashed, thin},
 pile/.style={thick, ->, >=stealth', shorten <=2pt, shorten
 >=2pt},
 every node/.style={color=black}
 ]

 \draw[arrow](-0.5,0.02)--(0.5,0.02);
 \draw[arrow](-0.5,-0.02)--(0.5,-0.02);
 \draw[arrow](0.5,0.02)--(0,0.5);
 \draw[arrow](0,0.5)--(-0.5,0.02);
 \draw[arrow](0.5,-0.02)--(0,-0.25);
 \draw[arrow](0,-0.25)--(-0.5,-0.02);
 \node at (-0.55,0){\tiny{$\circ$}};
 \node at (0.6,0){\tiny{$\bullet$}};
 \node at (0.05,0.6){\tiny{$S_1^{[p_1-1]}$}};
 \node at (0.06,-0.15){\tiny{$S_2^{[p_2-1]}$}};
 \node at (0,-0.35){\tiny{$\vdots$}};
 \draw[arrow](0.5,-0.02)--(0,-0.5);
 \draw[arrow](0,-0.5)--(-0.5,-0.02);
 \node at (0.05,-0.6){\tiny{$S_t^{[p_t-1]}$}};
 
 \draw[arrow](0,0.5)--(0.5,0.5);
 \node at (0.5,0.6){\tiny{$S_1^{[p_1-2]}$}};
 \draw[dashed line](0.5,0.5)--(1,0.5);
 \node at (1,0.6){\tiny{$S_1^{[2]}$}};
 \draw[arrow](1,0.5)--(1.5,0.5);
\node at (1.5,0.6){\tiny{$S_1^{[1]}$}};
 
 \draw[arrow](0,-0.25)--(0.5,-0.25);
 \node at (0.6,-0.15){\tiny{$S_2^{[p_2-2]}$}};
 \draw[dashed line](0.5,-0.25)--(1,-0.25);
 \node at (1,-0.15){\tiny{$S_2^{[2]}$}};
 \draw[arrow](1,-0.25)--(1.5,-0.25);
\node at (1.5,-0.15){\tiny{$S_2^{[1]}$}};
 
 \draw[arrow](0,-0.5)--(0.5,-0.5);
 \node at (0.5,-0.6){\tiny{$S_t^{[p_t-2]}$}};
 \draw[dashed line](0.5,-0.5)--(1,-0.5);
 \node at (1,-0.6){\tiny{$S_t^{[2]}$}};
 \draw[arrow](1,-0.5)--(1.5,-0.5);
\node at (1.5,-0.6){\tiny{$S_t^{[1]}$}};
 
 \node at (0.5, -0.35){\tiny{$\vdots$}};
 \node at (1, -0.35){\tiny{$\vdots$}};
\end{tikzpicture}

\end{figure}

The following fact can be founded in~\cite{bkl}*{Section 8}.
\begin{lemma}\label{l:weighted-sq}
Let $\mathbb{X}(\mathbf{p},\boldsymbol{\lambda})$ be a weighted projective line. There is a basic cluster-tilting object $T_{sq}$ such that its Gabriel quiver $Q_{T_{sq}}$ coincides with $Q_{\mathbf{p}}$.
\end{lemma}

 The following is an easy observation and we left the proof as an exercise.
\begin{lemma}\label{l:full-rank-weighted}
Let $\mathbf{p}=(p_1,\dots, p_t)$ be a weight sequence. If $p_i\geq 3$ and $(2,p_i)=1$ for all $1\leq i\leq t$, then the skew-symmetric matrix $B(Q_{\mathbf{p}})$ is of full rank.
\end{lemma}

\begin{theorem}\label{t:conj-weighted projective}
Conjecture~\ref{f-exchangeability2} holds true for cluster algebras arising from weighted projective lines.
\end{theorem}
\begin{proof}
According to Theorem~\ref{t:connected-exchange-graph} and Lemma~\ref{l:weighted-sq}, 
 it suffices to prove the result for a cluster algebra with initial exchange matrix $B(Q_{\mathbf{p}})$ for a weighted sequence $\mathbf{p}$. Similar to Theorem~\ref{t:conj-acyclic}, it suffices to prove the statement for a particular choice of coefficients.

Let $\mathbf{\hat{p}}=(\hat{p}_1,\dots, \hat{p}_t)$ be a weight sequence such that 
\[\hat{p}_i\geq 3,~\hat{p}_i\geq p_i~\text{and}~ (2, \hat{p}_i)=1~\text{for all $1\leq i\leq t$.}
\] 
In particular, the matrix $B(Q_{\mathbf{\hat{p}}})$ is of full rank by Lemma~\ref{l:full-rank-weighted}. We identify $Q_{\mathbf{p}}$ as a full subquiver of $Q_{\mathbf{\hat{p}}}$. We label the vertices of $Q_{\mathbf{\hat{p}}}$ by $\{1,\dots, n,n+1,\dots, m\}$ such that $\{1,\dots, n\}$ are precisely the vertices of $Q_{\mathbf{p}}$. In this way, 
\[B(Q_{\mathbf{\hat{p}}})=\begin{bmatrix}B(Q_{\mathbf{p}})&-C^{\top}\\ C&A
\end{bmatrix},
\]
where $C=(c_{ij})\in M_{(m-n)\times n}(\mathbb{Z})$ and $A\in M_{m-n}(\mathbb{Z})$.
 By Lemma~\ref{l:weighted-sq}, there is a basic cluster-tilting object $T_{sq}=\bigoplus_{i=1}^mT_i$ of $\mathcal{C}_{\mathbb{X}_{\mathbf{\hat{p}}}}$ such that it Gabriel quiver $Q_T$ is $Q_{\mathbf{\hat{p}}}$. Here we label the indecomposable direct summand $T_i$ according to the vertex $i$. Let $\mathcal{A}(B(Q_{\mathbf{\hat{p}}}))$ be the cluster algebra with trivial coefficients and initial seed $(\mathbf{x}=(x_1,\dots, x_m), B(Q_{\mathbf{\hat{p}}}) )$ and $\mathcal{A}(B(Q_{\mathbf{p}}))$ the cluster algebra over $\mathbb{P}:=\text{Trop}(x_{n+1}, \dots, x_m)$ with initial seed $(\mathbf{z}=(x_1,\dots, x_n), \mathbf{y}=(\prod_{i=1}^{m-n}x_{i+n}^{c_{i1}},\dots, \prod_{i=1}^{m-n}x_{i+n}^{c_{in}}),B(Q_{\mathbf{p}}))$. In particular, $\mathcal{A}_\bullet(B(Q_{\mathbf{p}}))$ is a subalgebra of $\mathcal{A}(B(Q_{\mathbf{\hat{p}}}))$. Similar to Proposition~\ref{pr:f-compatibility-symmetry}~(3), we identify the cluster pattern of $\mathcal{A}(B(Q_{\mathbf{p}}))$ with a subgraph of the cluster pattern of $\mathcal{A}(B(Q_{\mathbf{\hat{p}}}))$. We are going to prove the statement for the cluster algebra $\mathcal{A}(B(Q_{\mathbf{p}}))$.
 
 Applying Theorem~\ref{t:cluster-cat}, Theorem~\ref{t:connected-exchange-graph} and Proposition~\ref{p:cc-map}, the cluster character ${CC}_T(?)$ yields a bijection between the set of indecomposable rigid objects of $\mathcal{C}_{\mathbb{X}_{\mathbf{\hat{p}}}}$ and the set of cluster variables of $\mathcal{A}(B(Q_{\mathbf{\hat{p}}}))$. For an indecomposable rigid object $L\in \mathcal{C}_{\mathbb{X}_{\mathbf{\hat{p}}}}$, the cluster variable ${CC}_T(L)$ belong to $\mathcal{A}(B(Q_{\mathbf{p}}))$ if and only if $L\in \mathcal{Z}(U)$, where $U=\bigoplus_{j=n+1}^m T_{j}$ and $\mathcal{Z}(U)=\{X\in \mathcal{C}_{\mathbb{X}_{\mathbf{\hat{p}}}}|~\Hom_{\mathcal{C}_{\mathbb{X}_{\mathbf{\hat{p}}}}}(U, \Sigma X)=0\}$.
 
 Now assume that $L,N\in \mathcal{Z}(U)$ are indecomposable rigid objects such that 
 \[({CC}_T(L)\parallel {CC}_T(N))_{\mathcal{A}(B(Q_{\mathbf{p}}))}=({CC}_T(N)\parallel {CC}_T(L))_{\mathcal{A}(B(Q_{\mathbf{p}}))}=1.
 \]
By Proposition~\ref{pr:f-compatibility-symmetry}~(3), we have
\[({CC}_T(L)\parallel {CC}_T(N))_{\mathcal{A}(B(Q_{\mathbf{\hat{p}}}))}=({CC}_T(N)\parallel {CC}_T(L))_{\mathcal{A}(B(Q_{\mathbf{\hat{p}}}))}=1.
\]
Note that, as $B(Q_{\mathbf{\hat{p}}})$ is of full rank, we deduce that 
\[\dim_k\Hom_{\mathcal{C}_{\mathbb{X}_{\mathbf{\hat{p}}}}}(L,\Sigma N)=1
\]
by Theorem~\ref{p:categorical-f-degree}. Now the remaining proof is the same as the one of Theorem~\ref{t:conj-acyclic}.
\end{proof}

\subsection{Cluster algebras arising from marked surfaces}\label{subsection:markedsurface}
Let $(S, \mathbb{M})$ be a {\it marked surface}, \ie $S$ is a connected compact oriented Riemann surface with (possibly empty) boundary and $\mathbb{M}$ a non-empty finite set of marked points on $S$ with at least one marked point on each boundary component if $S$ has boundaries.
 A marked point in the interior of $S$ is a {\it puncture}. For technical reasons, $(S,\mathbb{M})$ is not a monogon with at most one puncture, a digon without punctures, a triangle without punctures nor a sphere with at most three punctures (cf.~\cite{fst}). The {\it tagged arc}, {\it tagged triangulation} and their {\it flips} of $(S,\mathbb{M})$ were introduced in~\cite{fst}. The {\it exchange graph} $\mathcal{E}_{(S,\mathbb{M})}$ of $(S, \mathbb{M})$ consists of tagged triangulations as vertices, while two vertices $T$ and $T'$ are connected by an edge if and only if $T'$ is obtained from $T$ by a flip.
 
 For a given tagged triangulation $T$ of $(S, \mathbb{M})$, Fomin, Shapiro and Thurston constructed a quiver $Q_T$ without loops nor $2$-cycles associated to $T$. Hence a cluster algebra $\mathcal{A}(B(Q_T))$ can be associated with a tagged triangulation $T$. 
 We refer to~\cite{fst} for the precise definition and construction. A cluster algebra $\mathcal{A}$ with exchange matrix $B(Q_T)$ for some tagged triangulation over certain marked surface $(S,\mathbb{M})$ is called a {\it cluster algebra arising from marked surfaces}.
 
 \begin{theorem}[\cite{fst}*{Theorem 7.11}~\cite{ft}*{Theorem 6.1}]\label{t:ms-arc-cv}Let $(S, \mathbb{M})$ be a marked surface and $T$ a tagged triangulation. Denote by $\mathcal{E}_T$ the connected component of $\mathcal{E}_{(S,\mathbb{M})}$ containing $T$.
 \begin{itemize}
 \item[(1)] If $(S,\mathbb{M})$ is a closed surface with exactly one puncture, then $\mathcal{E}_{(S, \mathbb{M})}$ has precisely two isomorphic connected components: one in which all ends of tagged arcs are plain and one in which they are notched. Otherwise, it is connected, {\it i.e.} any two tagged triangulation of $(S,\mathbb{M})$ are connected by a finite sequence of flips.
 
 \item[(2)] There is a bijection between the set of tagged arcs which belong to tagged triangulations contained in $\mathcal{E}_T$ and the set of cluster variables of the cluster algebra $\mathcal{A}(B(Q_T))$. It induces a bijection between the set of tagged triangulations lying in $\mathcal{E}_T$ and the set of clusters of $\mathcal{A}(B(Q_T))$. Moreover, a flip of tagged triangulations corresponds to a mutation of clusters.
 \end{itemize}
 \end{theorem} 
 
Let $T$ be a tagged triangulation and $Q_T$ the associated quiver. It follows from ~\cites{lab09,lab16,lad,tv} that there is a non-degenerate potential $W_T$ on $Q_T$ such that the associated Jacobian algebra $J(Q_T,W_T)$ is finite-dimensional. By applying Amiot's construction of generalized cluster category~\cite{Amiot}, we obtain a $2$-Calabi-Yau category $\mathcal{C}_{(Q_T,W_T)}$ with cluster-tilting objects. In particular, the tagged triangulation $T$ induces a basic cluster-tilting object $\Gamma$ in $\mathcal{C}_{(Q_T,W_T)}$. 
Moreover, the generalized cluster category $\mathcal{C}_{(Q_T,W_T)}$ is independent of the choice of the triangulation $T$, which we will denote it by $\mathcal{C}_{(S,\mathbb{M})}$.

By applying
Proposition \ref{p:cc-map} and Remark \ref{r:categorification}, the cluster character $CC_{\Gamma}(?)$ induces a bijection $\operatorname{cl}_\Gamma$ between the set of indecomposable rigid objects reachable from $\Gamma$ and the set of cluster variables of the cluster algebra $\mathcal{A}(B(Q_T))$. Combining the bijection in Theorem \ref{t:ms-arc-cv} with the inverse of $\operatorname{cl}_\Gamma$, we obtain a bijection $\phi_+$ from the set of tagged arcs which belong to triangulations lying in $\mathcal{E}_T$ and the set of indecomposable rigid objects reachable from $\Gamma$. We remark that the bijection $\phi_+$ induces a bijection between the set of triangulations lying in $\mathcal{E}_{T}$ and the set of cluster-tilting objects reachable from $\Gamma$ which commutes with flips and mutations. If $(S, \mathbb{M})$ is not a closed surface with exactly one puncture, we know that the cluster-tilting graph of $\mathcal{C}_{(S, \mathbb{M})}$ is connected by \cite{Yur19}*{Corollary 1.4} and the exchange graph of $(S, \mathbb{M})$ is also connected by Theorem \ref{t:ms-arc-cv}. In particular, $\phi:=\phi_+$ is a bijection from the set of tagged arcs of $(S, \mathbb{M})$ to the set of indecomposable rigid objects of $\mathcal{C}_{(S, \mathbb{M})}$.

Let us assume that $(S, \mathbb{M})$ is a closed surface with exactly one puncture. In this case, by \cite{Yur19}*{Corollary 1.4}, the cluster-tilting graph of $\mathcal{C}_{(S, \mathbb{M})}$ has precisely two connected components: $\mathcal{G}_\Gamma$ and $\mathcal{G}_{\Sigma \Gamma}$, where $\mathcal{G}_*$ stands for the connected component containing $*$. On the other hand, the exchange graph of $\mathcal{E}_{(S, \mathbb{M})}$ also has two connected components which we denote by $\mathcal{E}_{T_+}:=\mathcal{E}_{T}$ and $\mathcal{E}_{T_-}$.
Similar to the construction of $\phi_+$, there is a bijection $\phi_-$ from the set of tagged arcs which belong to triangulations lying in $\mathcal{E}_{T_-}$ to the set of indecomposable rigid objects reachable from $\Sigma \Gamma$ (cf. \cite{Yur19}*{Section 3 and 4}). Now for any tagged arc $l$, we define
\[\phi(l)=\begin{cases}
\phi_+(l)& \text{if $l$ belongs to a triangulation lying in $\mathcal{E}_{T_+}$;}\\
\phi_-(l)& \text{if $l$ belongs to a triangulation lying in $\mathcal{E}_{T_-}$;}
\end{cases}
\]

The following result is a refinement of \cite{Yur19}*{Corollary 1.4 and Section 3.2, 4.2} ({cf}.~\cite{qz} for marked surfaces with non-empty boundary).
\begin{proposition}\label{p:ms-connected}
Let $(S, \mathbb{M})$ be a marked surface. 
The map $\phi$ is a bijection from the set of tagged arcs of $(S, \mathbb{M})$ to the set of indecomposable rigid objects of $\mathcal{C}_{(S,\mathbb{M})}$. It induces a bijection between the set of triangulations of $(S, \mathbb{M})$ and the set of basic cluster-tilting objects of $\mathcal{C}_{(S,\mathbb{M})}$ which commutes with flips and mutations.
\end{proposition}
\begin{proof}
There is nothing to prove if $(S, \mathbb{M})$ is not a closed surface with exactly one puncture. Let us assume that $(S, \mathbb{M})$ is a closed surface with exactly one puncture. It remains to show that $\phi$ is a bijection.
We fix a triangulation $T$ of $(S, \mathbb{M})$ and denote by $\Gamma$ the induced basic cluster-tilting object in $\mathcal{C}_{(S,\mathbb{M})}$.
 
 By \cite{Yur19}*{Corollary 1.4}, each indecomposable rigid object of $\mathcal{C}_{(S, \mathbb{M})}$ is either reachable from $\Gamma$ or reachable from $\Sigma \Gamma$. It follows that $\phi$ is a surjection by definition. To prove that $\phi$ is a bijection, it suffices 
 to show that there does not exist an indecomposable rigid object $M$ which is reachable from $\Gamma$ and $\Sigma \Gamma$. Otherwise, there exists a basic cluster-tilting object $T_M$ (resp. $T_M'$) containing $M$ as a direct summand which is reachable from $\Gamma$(resp. $\Sigma \Gamma$). By applying the inverse of the suspension functor $\Sigma$, we know that $\Sigma^{-1}T_M'$ is reachable from $\Gamma$. Consequently, $\Sigma^{-1}T_M'$ is reachable from $T_M$. Let $T'$ be the triangulation corresponding to $\Sigma^{-1}T_M'$ and $\mathcal{A}(B(Q_{T'}))$ the associated cluster algebra. Denote by $x_M$ the cluster variable of $\mathcal{A}(B(Q_{T'}))$ corresponding to $M$. It follows from \cite{DK}*{Section 4} and \cite{dwz}*{Theorem 1.7} that the $g$-vector $g(x_M)$ of $x_M$ is $\operatorname{ind}_{\Sigma^{-1}T_M'}(M)=-[\Sigma^{-1}M]$ by identifying $\mathbb{Z}^n$ with $ \operatorname{G}_0(\add \Sigma^{-1}T_M')$. Applying \cite{Yur19}*{Theorem 1.2} to the triangulation $T'$, we obtain a contradiction. This completes the proof.
\end{proof}
As a byproduct, we obtain
\begin{corollary}\label{c:reach-closed-surface}
Let $(S, \mathbb{M})$ be a closed surface with exactly one puncture. Let $T$ and $T'$ be two cluster-tilting objects of $\mathcal{C}_{(S, \mathbb{M})}$. If $\add T\cap \add T'\neq \{0\}$, then $T$ is reachable from $T'$.
\end{corollary}

The $f$-vectors for cluster algebras associated to marked surfaces with principal coefficients have been investigated in~\cite{y}. Combining \cite{y}*{Theorem 1.8}, Theorem~\cite{qz}*{Theorem 1.1} and \cite{zzz}*{Theorem 3.4}, one obtains the following result:
\begin{theorem}\label{t:ms-f-vector}
Let $T$ be a tagged triangulation of a marked surface $(S, \mathbb{M})$. Let $\mathcal{A}_\bullet(B(Q_T))$ be the cluster algebra associated to $T$ with principal coefficients. Let $x,z$ be two cluster variables and $X,Z\in \mathcal{C}_{(S,\mathbb{M})}$ the corresponding indecomposable rigid objects, then
\[(x\parallel z)=\dim_k\Hom_{\mathcal{C}_{(S,\mathbb{M})}}(X,\Sigma Z).
\]
\end{theorem}
{\begin{remark}
Let $T$ be a tagged triangulation of a marked surface $(S, \mathbb{M})$ and $\mathcal{A}(B(Q_T))$ the cluster algebra associated to $T$. Let $x_\delta$ and $x_\gamma$ be the cluster variables of $\mathcal{A}(B(Q_T))$ corresponding to the tagged arcs $\delta$ and $\gamma$, respectively. We have 
\[(x_{\delta}\parallel x_{\gamma})=\operatorname{Int}(\delta, \gamma),
\]
where $\operatorname{Int}(\delta, \gamma)$ is the intersection number of $\delta$ and $\gamma$ in the sense of \cite{qz} (cf. \cite{y}*{Theorem 1.8}).
\end{remark}}

The following is the main result of this subsection.
\begin{corollary}\label{c:conj-marked surface}
Let $(S,\mathbb{M})$ be a marked surface and $T$ a tagged triangulation. Let $\mathcal{A}$ be the cluster algebra with initial exchange matrix $B(Q_T)$ associated to $T$. Then Conjecture~\ref{f-exchangeability2} holds true for $\mathcal{A}$.
\end{corollary}
\begin{proof}
Denote by $\Gamma$ the induced cluster-tilting object in $\mathcal{C}_{(S, \mathbb{M})}$. 

Assume that $(S, \mathbb{M})$ is not a closed surface with exactly one puncture. According to Proposition \ref{p:ms-connected} and Theorem \ref{t:ms-arc-cv}, the cluster character $CC_{\Gamma}(?)$ yields a bijection between the set of indecomposable rigid objects of $\mathcal{C}_{(S, \mathbb{M})}$ and the set of cluster variables of $\mathcal{A}$.
Now the result is a direct consequence of Theorem~\ref{t:exchangeable-2-CY} and Theorem~\ref{t:ms-f-vector}. 

Let us assume that $(S, \mathbb{M})$ is a closed surface with exactly one puncture. In this case,
the cluster character $CC_\Gamma(?)$ induces a bijection between the set of indecomposable rigid objects reachable from $\Gamma$ and the set of cluster variables of $\mathcal{A}(B(Q_T))$. 
Let $x,x'$ be two cluster variables of $\mathcal{A}(B(Q_T))$ such that 
\[(x\parallel x')=(x'\parallel x)=1.
\]
Denote by $X$ (resp. $X'$) the corresponding indecomposable rigid objects of $x$ (resp. $x'$).
In particular, there is a basic cluster-tilting object $T_x$ (resp. $T_{x'}$) reachable from $\Gamma$ which admits $X$ (resp. $X'$) as a direct summand. Let us rewrite $T_x$ as $T_x=\overline{T}_x\oplus X$.
By Theorem \ref{t:ms-f-vector}, we obtain that 
\[\dim_k\Hom_\mathcal{C}(X,\Sigma X')=1.
\]
It follows from Theorem \ref{t:exchangeable-2-CY} that $(X, X')$ is an exchange pair. In particular, there is a basic rigid object $M$ such that both $M\oplus X$ and $M\oplus X'$ are cluster-tilting objects of $\mathcal{C}_{(S,\mathbb{M})}$. 
We conclude that $M\oplus X$ is reachable from $\Gamma$ by Corollary \ref{c:reach-closed-surface} and we are done.

\end{proof}

\subsection{Counterexamples of exchangeability property for the $d$-vectors}\label{ss:counterexam}

The $d$-compatibility degree does not satisfy the analogous property of Conjecture \ref{f-exchangeability2}. Let us see some examples.
\begin{example}
We set $\PP=\{1\}$ the trivial semifield and consider a seed $(\xx,B)$, where 
\begin{align*}
\xx=(x_1,x_2,x_3,x_4,x_5), \quad
 B=
 \begin{bmatrix}
 0&1&0&-1&0\\
 -1&0&1&1&-1\\
 0&-1&0&1&0\\
 1&-1&-1&0&1\\
 0&1&0&-1&0\\
 \end{bmatrix}.
\end{align*}
This is of type $\hat{D}_4$. Moreover, we set 
\begin{align*}
\xx'=(x_1',x_2',x_3',x_4',x_5)=\mu_4\mu_3\mu_2\mu_1(\xx).
\end{align*}
Then, we have $(x_1 \parallel x'_4)_d=(x'_4 \parallel x_1)_d=1$ and $(x_1 \parallel x'_4)=(x'_4 \parallel x_1)=2$. Let us see this fact by using marked surface and their flips. $\Acal(B)$ is a cluster algebra arising from the marked surface in Figure \ref{Bsurface2}.
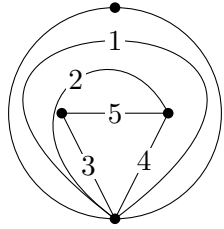
\begin{figure}[ht]
 \centering
 \caption{Marked surface corresponding to $B$ \label{Bsurface2}}
\begin{tikzpicture}[baseline=0mm,scale=0.7]
 \coordinate (d) at (0,-2);
 \coordinate (u) at (0,2);
 \coordinate (0) at (0,0);
 \coordinate (cl) at (180:1);
 \coordinate (cr) at (0:1);
 \draw (0) circle (2);
 \draw (d) to node[fill=white,inner sep=1]{$3$} (cl);
 \draw (d) to node[fill=white,inner sep=1,pos=0.55]{$4$} (cr);
 \draw (cl) to node[fill=white,inner sep=1,pos=0.5]{$5$} (cr);
 \draw (d) .. controls (-6,2.5) and (6,2.5) .. 
 node[fill=white,inner sep=1,pos=0.5]{$1$} (d);
 \draw (d) .. controls (180:2.7) and (90:2) .. 
 node[fill=white,inner sep=1,pos=0.55]{$2$} (cr);
 \fill(d) circle (1mm); \fill(u) circle (1mm); \fill(cl) circle (1mm); \fill(cr) circle (1mm);
 
\end{tikzpicture}
\end{figure}
We consider flipping the marked surface in Figure \ref{Bsurface2} at 1,2,3,4. The relative position of arc corresponding to $x_1$ and $x'_4$ is as in Figure \ref{twovariableposition2}. 

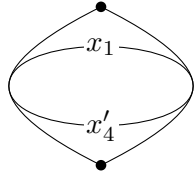
\begin{figure}[ht]
 \centering
 \caption{Relative position of arc corresponding to $x_1$ and $x'_4$\label{twovariableposition2}}
\begin{tikzpicture}[baseline=0mm,scale=0.7]
 \coordinate (d) at (0,-3);
 \coordinate (0) at (0,0);
 \draw (d) .. controls (-6,0) and (6,0) .. 
 node[fill=white,inner sep=1,pos=0.5]{$x_1$} (d);
 \draw (0) .. controls (-6,-3) and (6,-3) .. 
 node[fill=white,inner sep=1,pos=0.5]{$x'_4$} (0);
 \fill(d) circle (1mm); \fill(0) circle (1mm);
\end{tikzpicture}
\end{figure}
Considering the intersection number induced by the $d$-vector (\cite{fst}*{Definition 8.4}), we have $(x_1 \parallel x'_4)_d=(x'_4 \parallel x_1)_d=1$. On the other hand, considering the intersection number induced by the $f$-vector (\cite{y}*{Section 1}), we have $(x_1 \parallel x'_4)=(x'_4 \parallel x_1)=2$. Therefore, by Corollary \ref{c:conj-marked surface}, $x_1$ and $x'_4$ are not exchangeable. This example implies $(\cdot \parallel \cdot)_d$ does not satisfy the similar property of Conjecture \ref{f-exchangeability2}. We remark that $\Acal(B)$ satisfies the similar property of Proposition \ref{pr:f-compatibility-symmetry} for the $d$-vectors because of Theorem \ref{D surface} (3).
\end{example}

\begin{example}
We set $\mathbb{P}=\{1\}$ and consider a seed $(\mathbf{x}, B)$, where
\begin{align*}
\mathbf{x}=(x_1.x_2.x_3), \quad B=\begin{bmatrix}
0&2&-1\\
-2&0&1\\
1&-1&0
\end{bmatrix}.
\end{align*}
We fix a cluster pattern by assigning $\Sigma_{t_0}=(\mathbf{x}, B)$ to the rooted vertex $t_0$ of $\mathbb{T}_3$. The cluster algebra $\mathcal{A}(B)$ is acyclic (type $\hat{A}_2$).
Let 
\[\xymatrix{t_0\ar@{-}[r]^3&t_1\ar@{-}[r]^2&t_2\ar@{-}[r]^1&t_3}
\]
be a subgraph of $\mathbb{T}_3$. According to \cite{fk}*{Example 6.7}, we have
\[\mathbf{f}_{1,t_3}^{B,t_0}=\begin{bmatrix}1\\ 1\\ 2\end{bmatrix}, \quad \mathbf{d}_{1;t_3}^{B,t_0}=\begin{bmatrix}1\\ 1\\1
\end{bmatrix}.
\]
Therefore $(x_3\parallel x_{1;t_3})=(x_{1;t_3}\parallel x_3)=2$ by Proposition~\ref{pr:f-compatibility-symmetry} (2). Hence $x_3$ and $x_{1;t_3}$ are not exchangeable by Theorem~\ref{t:conj-acyclic}.
On the other hand, a direct computation shows that 
\[\mathbf{d}_{3,t_0}^{B_{t_3};t_3}=\begin{bmatrix}1\\1\\1
\end{bmatrix}.
\]
Hence $(x_3\parallel x_{1;t_3})_d=(x_{1,t_3}\parallel x_3)_d=1$. In particular, $(\cdot\parallel \cdot)_d$ does not satisfy the similar property of Conjecture~\ref{f-exchangeability2}.

\end{example}

\bibliography{myrefs}
\end{document}